\documentclass[12pt]{amsart}
\usepackage{pifont}
\usepackage{txfonts}
\usepackage{}
\usepackage{amsfonts}
\usepackage{graphicx}
\usepackage{epstopdf}
\usepackage{amsmath}
\usepackage{amsthm}
\usepackage{latexsym,bm}
\usepackage{amssymb}
\usepackage{esint}
\usepackage{enumitem}
\usepackage{indentfirst}
\usepackage{amscd}
\newtheorem{thm}{Theorem}[section]
\newtheorem{cor}[thm]{Corollary}
\newtheorem{lem}[thm]{Lemma}
\newtheorem{prop}[thm]{Proposition}
\newtheorem{defn}[thm]{Definition}
\numberwithin{equation}{section}
\numberwithin{Remark}{section}

\begin{document}

\title{A flow approach to the generalized Loewner-Nirenberg problem of the $\sigma_k$-Ricci equation}

\renewcommand{\subjclassname}{\textup{2000} Mathematics Subject Classification}
 \subjclass[2010]{Primary 53C25; Secondary 58J05, 53C30, 34B15}

\begin{abstract}
We introduce a flow approach to the generalized Loewner-Nirenberg problem $(\ref{equn_sigmak})-(\ref{equn_boundaryblowingup})$ of the $\sigma_k$-Ricci equation on a compact manifold $(M^n,g)$ with boundary. We prove that for initial data $u_0\in C^{4,\alpha}(M)$ which is a subsolution to the $\sigma_k$-Ricci equation $(\ref{equn_sigmak})$, the Cauchy-Dirichlet problem $(\ref{equn_flowsigmak})-(\ref{equn_flowboundarydata})$ has a unique solution $u$ which converges in $C^4_{loc}(M^{\circ})$ to the solution $u_{\infty}$ of the problem $(\ref{equn_sigmak})-(\ref{equn_boundaryblowingup})$, as $t\to\infty$.
\end{abstract}

\renewcommand{\subjclassname}{\textup{2000} Mathematics Subject Classification}
 \subjclass[2010]{Primary 53C44; Secondary 35K55, 35R01, 53C21}


\thanks{$^\dag$ Research partially supported by the National Natural Science Foundation of China No. 11701326 and  the Young Scholars Program of Shandong University 2018WLJH85.}

\address{Gang Li, Department of Mathematics, Shandong University, Jinan, Shandong Province, China}
\email{runxing3@gmail.com}

\maketitle


\section{Introduction}

Let $(M^n,g)$ be a compact Riemannian manifold with boundary of dimension $n\geq3$. Denote $M^{\circ}$ to be the interior of $M$. In \cite{LiGang}, we considered the Cauchy-Dirichlet problem of the Yamabe flow which starts from a positive subsolution of the Yamabe equation $(\ref{equn_scalarcurvatureequn})$ and converges in $C_{loc}^2(M^{\circ})$ to the solution to the Loewner-Nirenberg problem
\begin{align}
&\label{equn_scalarcurvatureequn}\frac{4(n-1)}{n-2}\Delta u-R_gu-n(n-1)u^{\frac{n+2}{n-2}}=0,\,\,\,\text{in}\,\,M^{\circ},\\
&\label{equn_blowingupboundarydata}u(p)\to\infty,\,\,\,\text{as}\,\,p\to\partial M,
\end{align} 
which is originally studied by Loewner and Nirenberg \cite{LL} on Euclidean domains, and later by Aviles and McOwen \cite{AM}\cite{AM1} on general compact manifolds with boundary. A signature feature of our flow is that it preserves the solution $u(\cdot,t)$ as a sub-solution to the Yamabe equation for $t>0$. In this paper, we extend this approach to study the generalized Loewner-Nirenberg problem for the fully nonlinear equation studied in \cite{Guan1} and \cite{GSW}.

\begin{defn}
For $(\lambda_1,\,...,\,\lambda_n)\in \mathbb{R}^n$ and $k=1,..,n$, we define the elementary symmetric functions as
\begin{align*}
\sigma_k(\lambda_1,\,...,\,\lambda_n)=\sum_{i_1<...<i_k}\lambda_{i_1}\,...\,\lambda_{i_k},
\end{align*}
and define the cone
\begin{align*}
\Gamma_k^+=\{\Lambda=(\lambda_1,\,...,\,\lambda_n)\in \mathbb{R}^n\,\big|\,\sigma_j(\Lambda)>0,\,\forall j\leq k\},
\end{align*}
which is the connected component of the set $\{\sigma_k>0\}$ containing the positive definite cone on $\mathbb{R}^n$. We also define $\Gamma_k^-=-\Gamma_k^+$. For a symmetric $n\times n$ matrix $A$, $\sigma_k(A)$ is defined to be $\sigma_k(\Lambda)$ with $\Lambda=(\lambda_1,\,...,\,\lambda_n)$ the eigenvalues of $A$.
\end{defn}
The $\sigma_k$-scalar curvature equation is introduced in \cite{Viaclovsky}. Let $(M^n,g)$ be a smooth compact Rimeannian manifold with boundary of dimension $n\geq3$. Denote $Ric_g$ as the Ricci tensor of $g$. In \cite{GSW}, for any $k=1,...,n$, the authors studied the Dirichlet boundary value problem of the $\sigma_k$ equation of $-Ric_g$, in seek of a conformal metric $\bar{g}=e^{2u}g$ such that $Ric_{\bar{g}}\in \Gamma_k^-$ and
\begin{align}
&\sigma_k(-Ric_{\bar{g}})\equiv \sigma_k(-\bar{g}^{-1}Ric_{\bar{g}})=\bar{\beta}_{k,n}\,\,\,\text{in}\,\,M,\\
&\label{equn_bdvzero}u=0\,\,\,\,\text{on}\,\,\partial M,
\end{align}
where $\bar{\beta}_{k,n}=(n-1)^k\left(\begin{matrix} n\\ k\end{matrix}\right)$, or equivalently, we have the {\it $\sigma_k$-Ricci equation}
\begin{align}\label{equn_sigmak}
\sigma_k(\bar{\nabla}^2u)=\bar{\beta}_{k,n}e^{2ku}
\end{align}
 where
 \begin{align}\label{equn_thetensor}
 \bar{\nabla}^2 u&=-Ric_g+ (n-2)\nabla^2u+\Delta u\,g+(n-2)(|d u|^2g-du \otimes d u).
 \end{align}
 A more interesting result in \cite{GSW} is that they generalized the Loewner-Nirenberg problem to the $\sigma_k$-Ricci equation $(\ref{equn_sigmak})$ (see also \cite{Guan1}). They proved that there exists a unique solution $u_k$ to $(\ref{equn_sigmak})$ with the property that
 \begin{align}\label{equn_boundaryblowingup}
 u_k(p)\to+\infty
 \end{align}
uniformly as $p\to \partial M$; moreover,
\begin{align}\label{equn_limitboundaryorder1}
\lim_{p\in \partial M}[u_k(p)+\log(r(p))]=0
\end{align}
as $p\to \partial M$, where $r(p)$ is the distance of $p$ to $\partial M$. Notice that in \cite{Guan1} Guan gave an alternative approach to similar results, using metrics of negative Ricci curvature in the conformal class constructed in \cite{Lohkamp} as the background metric. In comparison, the argument in \cite{GSW} uses a general background conformal metric and concludes the existence of a prescribed $\sigma_n$-Ricci curvature metric of negative Ricci curvature. In this paper, we give a flow approach to the generalized Loewner-Nirenberg problem to the $\sigma_k$-Ricci equation $(\ref{equn_sigmak})$ starting from a sub-solution to $(\ref{equn_sigmak})$, with a background metric of negative Ricci curvature in the conformal class.
In particular, we introduce the Cauchy-Dirchlet problem $(\ref{equn_flowsigmak})-(\ref{equn_flowboundarydata})$ of the $\sigma_k$-Ricci curvature flow.

In order to get the lower bound control of the blowing up ratio near the boundary, we need to assume that the boundary data $\phi$ could not go to infinity too slowly as $t\to\infty$.

\begin{defn}\label{defn_lowspeedincreasingfunction}
We call a function $\xi(t)\in C^1([0,\infty))$ a low-speed increasing function if, $\xi(t)>0$ for $t\geq 0$, $\lim_{t\to\infty}\xi(t)=\infty$, and there exist two constants $T>0$ and $\tau>0$ such that
for $t\geq T$,
\begin{align}\label{inequn_lowspeedincreasing}
\xi'(t)\leq \tau.
\end{align}
\end{defn}
Here are some examples of low-speed increasing functions: $t^{\alpha}$ for some $0<\alpha<1$, $\log(t)$, and finitely many composition of $\log$ functions: $\log\circ\log\circ...\circ\log(t)$ for $t>0$ large, etc.

\begin{thm}\label{thm_main}
Assume $(M^n, g)$ $(n\geq3)$ is  a compact manifold with boundary of $C^{4,\alpha}$, and $(M,g)$ is either a compact domain in $\mathbb{R}^n$ or with Ricci curvature
 $Ric_g\leq -\delta_0g$ for some $\delta_0\geq (n-1)$. Assume $u_0\in C^{4,\alpha}(M)$ is a subsolution to $(\ref{equn_sigmak})$
    satisfying $(\ref{inequn_c4compatiblebdincreasing})$ at the points $x\in \partial M$ where $v(x)=0$ for the function $v$ defined in $(\ref{equn_defnfunctionA1})$.
    Also, assume $\phi\in C^{4+\alpha,2+\frac{\alpha}{2}}(\partial M\times[0,T_0])$ for all $T_0>0$, $\phi_t(x,t)\geq 0$ on $\partial M\times[0,+\infty)$ and
     $\phi$ satisfies the compatible condition $(\ref{equn_compatibleequations})$ with $u_0$. Moreover,  assume that there exist a low-speed increasing function $\xi(t)$ satisfying
 $(\ref{inequn_lowspeedincreasing})$ for some $T>0$ and $\tau>0$, and a constant $T_1>T$ such that $\phi(x,t)\geq \log(\xi(t))$ for
 $(x,t)\in \partial M\times[T_1,\infty)$. Then there exists a unique solution $u\in C^{4,2}(M\times [0,+\infty))$ to the Cauchy-Dirichlet problem $(\ref{equn_flowsigmak})-(\ref{equn_flowboundarydata})$ such that
      $u\in C^{4+\alpha,2+\frac{\alpha}{2}}(M\times[0,T])$ for all $T>0$.
Moreover, the solution $u$ converges to a solution $u_{\infty}$ to the equation $(\ref{equn_sigmak})$ locally uniformly on $M^{\circ}$ in $C^4$, and
\begin{align*}
\lim_{x\to \partial M} \,(u_{\infty}(x)+\log(r(x)))=0,
\end{align*}
where $r(x)$ is the distance of $x$ to $\partial M$.
\end{thm}
Notice that our assumption on the boundary data $\phi$ and the speed that $\phi\to\infty$ as $t\to\infty$ is pretty general.  When $u_0$ is a solution to $(\ref{equn_sigmak})$ in a neighborhood of $\partial M$, then $(\ref{inequn_c4compatiblebdincreasing})$ holds automatically; while the condition $(\ref{inequn_c4compatiblebdincreasing})$ disappears when $u_0$ is a strict sub-solution to $(\ref{equn_sigmak})$ in a neighborhood of $\partial M$. For instance, for any sub-solution $u_0$ to $(\ref{equn_sigmak})$, $u_0-C$ is a strict sub-solution for any constant $C>0$. For the long time existence of the flow, one needs to establish the global a priori estimates on the solution $u$ up to $C^2$-norm: both the boundary estimates and the interior estimates, starting from the $L^{\infty}$ control by the maximum principle and heavily depending on the monotonicity of $u$ and the control of $u_t$. In particular, $u_t\geq0$ and hence $u(\cdot,t)$ is a sub-solution to $(\ref{equn_sigmak})$ for any $t\geq0$, which together with the uniform interior upper bound control makes the convergence possible and gives a natural lower bound of $u$. For the convergence of the flow, we have to give the uniform interior $C^2$-estimates on $u$ which is independent of $t>0$. Finally the asymptotic boundary behavior near the boundary as $t\to\infty$ is established, which implies that the limit function is a solution to the generalized Loewner-Nirenberg problem.  Many of the barrier functions in these estimates can be viewed as a parabolic version of those in \cite{GSW} and \cite{Guan1}. This flow approach works well for the Loewner-Nirenberg problem of more general nonlinear euqations in \cite{Guan1}.

\begin{cor}
Assume $(M^n, g)$ is  a compact manifold with boundary of $C^{4,\alpha}$. Then there exists a sub-solution $u_0$ to $(\ref{equn_sigmak})$ and a $\sigma_k$-Ricci curvature flow $g(t)=e^{2u}g$ starting from $g_0=e^{2u_0}g$ and satisfying $(\ref{equn_flowsigmak})$ and the Cauchy-Dirichlet condition $(\ref{equn_flowinitialdata})-(\ref{equn_flowboundarydata})$ with some boundary data $\phi$ such that $g(t)$ converges to $g_{\infty}=e^{2u_{\infty}}g$ locally uniformly in $C^{4}$ as $t\to+\infty$, where
$u_{\infty}$ is the unique generalized Loewner-Nirenberg solution to $(\ref{equn_sigmak})$ i.e., $u_{\infty}(x)\to \infty$ as $x \to \partial M$. Moreover,
\begin{align*}
\lim_{x\to \partial M}\,(u_{\infty}(x)+\log(r(x)))=0.
\end{align*}
\end{cor}
\begin{proof}
As discussed in Section \ref{section_sub1.1}, by \cite{Lohkamp} there exists a metric in the conformal class $[g]$ of $C^{4,\alpha}$, which is still denoted as $g$ such that $Ric_g<-(n-1)g$. If $M$ is a Euclidean domain, we can alternatively just choose $g$ to be the Euclidean metric. We then take $g$ as the background metric. Now we choose a sub-solution $u_0$ to $(\ref{equn_sigmak})$ such that $u_0$ satisfies $(\ref{inequn_c4compatiblebdincreasing})$ on the boundary.
For instance, if $(M, g)$ is a sub-domain in Euclidean space, we choose $u_0$ to be either the global sub-solution constructed in \cite{GSW} (just take $\eta(s)= s$ for the subsolution $\underline{u}$ in Section \ref{section_sub1.1}) for the constants $A$ and $p$ large, or the solution to $(\ref{equn_sigmak})$ with $u_0=0$ on $\partial M$ obtained in \cite{GSW} or \cite{Guan1}. For general $(M,g)$, with the background metric $g$ satisfying $Ric_g<-(n-1)g$, we can either take $u_0$ to be the solution to $(\ref{equn_flowsigmak})$ with $u_0=0$ on $\partial M$ obtained in \cite{GSW} or \cite{Guan1}, or use the global sub-solution constructed in Section \ref{section_sub1.1}, or $u_0=v-1$ where $v\in C^{4,\alpha}(M)$ is any sub-solution of $(\ref{equn_sigmak})$ and hence $u_0$ is a strict sub-solution (with "$>$" instead of "$=$" in $(\ref{equn_flowsigmak})$). Then we construct the boundary data $\phi\in C^{4,2}(\partial M\times[0,\infty))$
 satisfying the compatible condition $(\ref{equn_compatibleequations})$ at $t=0$ such that $\phi\in C^{4+\alpha,2+\frac{\alpha}{2}}(\partial M\times[0,T])$ for any $T>0$,
  $\phi_t\geq 0$ on $\partial M \times[0,\infty)$ and $\phi(x,t)\geq \xi(t)$ on $\partial M\times[T,\infty)$ for some $T>0$, where $\xi(t)$ is a low-speed increasing function in Definition \ref{defn_lowspeedincreasingfunction}. Now we consider the solution to the Cauchy-Dirichlet boundary value problem $(\ref{equn_flowsigmak})-(\ref{equn_flowboundarydata})$. Therefore, by Theorem \ref{thm_main}, we have the required conclusion.
\end{proof}
One can easily adapt this approach to the convergence of a $\sigma_k$-Ricci curvature flow to the solution to the Dirichlet boundary value problem of $(\ref{equn_sigmak})$.

\begin{cor}\label{cor_Dirichletboundaryvalueproblemflow}
Assume $(M^n, g)$ is  a compact manifold with boundary of $C^{4,\alpha}$. Let $\varphi_0\in C^{4,\alpha}(\partial M)$. Then there exists a sub-solution $u_0$ to $(\ref{equn_sigmak})$ and a $\sigma_k$-Ricci curvature flow $g(t)=e^{2u}g$ starting from $g_0=e^{2u_0}g$ and satisfying $(\ref{equn_flowsigmak})$ and some Cauchy-Dirichlet condition such that $g(t)$ converges to $g_{\infty}=e^{2u_{\infty}}g$ uniformly in $C^{4}$ as $t\to+\infty$, where
$u_{\infty}$ is the unique solution to $(\ref{equn_sigmak})$ such that $u_{\infty}=\varphi_0$ on $\partial M$.
\end{cor}

Recently,  in \cite{GLN} the authors studied a more general fully nonlinear equations with less restriction on regularity and convexity on the nonlinear structures on smooth domains in Euclidean space and obtained a unique continuous viscosity solution, which is locally Lipschitz in the interior and shares the same blowing up ratio with the  solution to the Loewner-Nirenberg problem near the boundary.

The paper is organized as follows: In Section \ref{section_sub1.1}, we construct a global sub-solution in  $C^{4,\alpha}(M)$ to the $\sigma_k$-Ricci equation $(\ref{equn_sigmak})$. In Section \ref{section_C2estimates}, we formulate the maximum principle, show the monotonicity of the flow and give the global a priori estimates of the solution for the long time existence of the flow. In Section \ref{section_convergenceoftheflow}, we first prove the long time existence of the flow based on the a priori estimates in Section \ref{section_C2estimates}, and then we give the uniform interior estimates of the solution independent of $t$, and establish the asymptotic behavior of the solution near the boundary (see Lemma \ref{lem_boundarylowerbd1}) and prove Theorem \ref{thm_main}. Finally we give a proof of Corollary \ref{cor_Dirichletboundaryvalueproblemflow}.

{\bf Acknowledgements.} The author would like to thank Professor Matthew Gursky for helpful discussion and Professor Jiakun Liu for nice talks on nonlinear equations.

\section{A global subsolution to $(\ref{equn_sigmak})$}\label{section_sub1.1}

We now construct a global subsolution $\underline{u}\in C^{4,\alpha}(M)$ to the homogeneous Dirichlet boundary value problem $(\ref{equn_bdvzero})-(\ref{equn_sigmak})$. Recall that in \cite{GSW}, the authors constructed a global subsolution with singularity at the cut locus of the distance function to some point, which serves as a global uniform lower bound of the solution. We modify it to a smooth function in order to avoid complicated argument on the cut locus in our setting. Let $(M^n,g)$ be a compact Riemannian manifold with boundary of $C^{4,\alpha}$. We extend the manifold to a new manifold with boundary $M_1=M\bigcup (\partial M\times[0,\epsilon_0])$ for some small constant $\epsilon_0$ with $\partial M=\partial M\times \{0\}$
and extend $g$ to a $C^{4,\alpha}$ metric on $M_1$. One can construct a conformal metric $h\in [g_1]$ of $C^{4,\alpha}$ with $Ric_h<0$ on $M_1$, which always exists by the proof
 in \cite{Lohkamp}. Without loss of generality,
  we take $h$ as the background metric and still denote $h$ as $g$ in $M_1$, with $Ric_g\leq -\delta_0 g$ for some constant $\delta_0>0$ in $M$. In fact by scaling we assume $Ric_g\leq -\delta_0 g$ with $\delta_0>(n-1)$ large in $M$.

  Notice that there exist two small constants $0<\epsilon_1<\delta$ such that $\text{dist}(x,\partial M_1)>2\epsilon_1+4\delta$ for $x\in \partial M$, and also $\epsilon_1+2\delta$ is less than the injectivity radius of any point $q$ in the tubular neighborhood of $\partial M$
  \begin{align*}
 \Omega=\{ x\in M_1\,\big|\,\text{dist}_g(x,\partial M)\leq \epsilon_1+2\delta\},
  \end{align*}
 with $\text{dist}_g(x,\partial M)$ distance function to $\partial M$, and moreover  for $x\in\Omega$, the distance $\text{dist}_g(x,\partial M)$ is realized by a unique point $x_1\in\partial M$ through a unique shortest geodesic connecting $x$ and $x_1$, which is orthogonal to $\partial M$ at $x_1$.  For any $x_0\in \partial M$, we pick up the point $\bar{x}\in M_1\setminus M$ on the geodesic starting from $x_0$ along the outer normal vector of $\partial M$ so that $\text{dist}_g(x_0,\bar{x})=\epsilon_1$. 
 We define the distance function $r(x)=dist_g(x,\bar{x})$ for $x\in M_1$. In particular, $r(x_0)=\epsilon_1$ and $r$ is smooth for $r\leq 2\delta+\epsilon_1$. It is clear that $r(x)\geq r(x_0)$ for any $x\in M$ and the equality holds if and only if $x=x_0$. 

 Now for a fixed $x_0\in \partial M$ and the corresponding point $\bar{x}$, we can choose the subsolution in the following way: We let $A>0$ and $p>0$ be two large constant to be determined so that
 \begin{align*}
 N=A[-(\delta+r(x_0))^{-p}+r(x_0)^{-p}]
 \end{align*}
  is large, and we define a convex
     function $\eta \in C^5(\mathbb{R})$, so that
     \begin{align*}
     &\eta(s)=\eta(A(2\delta+r(x_0))^{-p}-r(x_0)^{-p}))\,\, \text{for}\,\, s\leq A\,[(2\delta+r(x_0))^{-p}-r(x_0)^{-p}],\,\,\text{and}\\
     &\eta(s)=s,\,\,\,\text{for}\,\,s\geq A\,[(\delta+r(x_0))^{-p}-r(x_0)^{-p}].
     \end{align*}
It is clear that $\eta'(s)\geq 0$ and $\eta''(s)\geq 0$, for $s\in \mathbb{R}$. Now we define
\begin{align*}
\underline{u}(x)=\eta(A\,(r(x)^{-p}-r(x_0)^{-p})),
\end{align*}
and hence $\underline{u}\in C^{4,\alpha}(M)$. We {\bf claim} that we can choose uniform large constants $A>0$ and $p>0$ independent of $x_0\in \partial M$ so that $\underline{u}$ is a subsolution. First, we give the calculation
\begin{align*}
&\nabla \underline{u}(x)= -Ap\eta'\,r^{-p-1}\nabla r,\\
&\nabla_i \nabla_j \underline{u}(x)=A^2p^2 \eta''\,r^{-2p-2} \nabla_ir \nabla_jr+p(p+1)A\eta'\,r^{-p-2}\nabla_i r \nabla_j r-pA\eta'\,r^{-p-1}\nabla_i \nabla_j r\\
&\,\,\,\,\,\,\,\,\,\,\,\,\,\,\,\,\,\,\,\,\,=A^2p^2 \eta''\,r^{-2p-2} \nabla_ir \nabla_jr+Ap r^{-p-2}\eta'[(p+1)\,\nabla_i r \nabla_j r-\,r\nabla_i \nabla_j r],\\
&\Delta \underline{u}(x)=A^2 p^2 \eta''\,r^{-2p-2} \big|\nabla r\big|^2 +Ap(p+1)\eta'\,r^{-p-2}\big|\nabla r\big|^2-Ap\eta'\,r^{-p-1}\Delta r\\
&\,\,\,\,\,\,\,\,\,\,\,\,\, \,=A^2 p^2 \eta''\,r^{-2p-2}  +Ap r^{-p-2}\eta'[(p+1)\, -\,r\Delta r],
\end{align*}
It is clear that for given $\delta>\epsilon_1>0$, we can choose $p>0$ such that, for any $x\in M$ such that $r(x)\leq 2\delta+r(x_0)$, we have that
$(p+1)\, -\,r\Delta r>0$, where $p>0$ is independent of the choice of $x_0\in \partial M$. In fact, we choose $p>0$ large so that the matrix
\begin{align*}
[(p+1\, -\,r\Delta r) g_{ij}-(n-2) r\nabla_i\nabla_j r]
\end{align*}
is positive for $x\in M$ such that $r(x_0)\leq r(x)\leq 2\delta+r(x_0)$. Therefore,
\begin{align}\label{term1_positivity}
(n-2)\nabla^2\underline{u}(x)+\Delta \underline{u}(x)g
\end{align}
is always non-negative on $M$. Since $-Ric>\delta_0g$ with some constant $\delta_0>(n-1)$ on $M$ and
\begin{align*}
|d \underline{u}(x)|^2g-d\underline{u}(x) \otimes d \underline{u}(x)
\end{align*}
 is semi-positive, we have that for $0\leq s\leq 1$,
\begin{align*}
\bar{\nabla}_s^2 \underline{u}(x)&\equiv sg-(1-s)Ric_g+ (n-2)\nabla^2\underline{u}(x)+\Delta \underline{u}(x)+(n-2)(|d \underline{u}(x)|^2g-d\underline{u}(x) \otimes d \underline{u}(x))\\
&\geq (s+(1-s)\delta_0) g\geq g.
\end{align*}
By the definition of $\eta$,
 \begin{align*}
 \underline{u}(x)\leq \eta(A((r(x_0)+\delta)^{-p}-r(x_0)^{-p}))=A((r(x_0)+\delta)^{-p}-r(x_0)^{-p})
 \end{align*}
for $r(x)\geq \delta+r(x_0)$. Now $A>0$ and $p>0$ is chosen to be large so that
\begin{align*}
A((r(x_0)+\delta)^{-p}-r(x_0)^{-p})<-\frac{1}{2}\log((n-1)),
\end{align*}
and hence
\begin{align}\label{inequn_subsolutionstrictfurways}
\sigma_n(g^{-1}\bar{\nabla}_s^2 \underline{u})\geq \sigma_n(\delta_i^j) =1>\bar{\beta}_{n,n}e^{2n\underline{u}}
\end{align}
for $x\in M$ with $r(x)\geq \delta+r(x_0)$. On the other hand, for $x\in M$ with $r(x)\leq \delta+r(x_0)$, we have
 \begin{align*}
 &\eta(A(r(x)^{-p}-r(x_0)^{-p}))=A(r(x)^{-p}-r(x_0)^{-p}),\\
 &\eta'(A(r(x)^{-p}-r(x_0)^{-p}))=1,\\
 &\eta''(A(r(x)^{-p}-r(x_0)^{-p}))=0,
 \end{align*}
and hence, as discussed in \cite{GSW}, for $A>0$ and $p>0$ large,
 \begin{align}\label{inequn_strictsubsolutioncomparison}
 \bar{\nabla}^2\underline{u}(x)> (n-1)g,
 \end{align}
 for $x\in M$ with $r(x)\leq \delta+r(x_0)$, where the term $(\ref{term1_positivity})$ serves as the main controlling positive term. Since $\underline{u}\leq 0$, we have $\underline{u}\in C^{4,\alpha}(M)$ is a subsolution to the $\sigma_n$ equation when $r(x)\leq \delta+r(x_0)$ and hence a sub-solution on $M$ by $(\ref{inequn_subsolutionstrictfurways})$, with $\underline{u}\leq 0$ on $\partial M$. Let $S_k=\sigma_k(\bar{\nabla}^2 u){\left(\begin{matrix} n\\ k\end{matrix}\right)}^{-1}$ for $1\leq k\leq n$.  
 By Maclaurin's inequality,
\begin{align*}
S_1\geq S_2^{\frac{1}{2}}\geq..\geq S_k^{\frac{1}{k}}\geq..\geq S_n^{\frac{1}{n}},
\end{align*}
which implies that a subsolution to the $\sigma_n$ equation is a subsolution to the $\sigma_k$ equation for $1\leq k\leq n$, while a supersolution of the $\sigma_1$ equation such that $\bar{\nabla}^2u\in \Gamma_k^+$ is a supersolution to the $\sigma_k$ equation for $1\leq k\leq n$. In particular, $\underline{u}$ serves as a subsolution to the $\sigma_k$ equations and a uniform lower bound of the solutions to the homogeneous Dirichlet boundary value problem for $1\leq k\leq n$.
Moreover, by $(\ref{inequn_subsolutionstrictfurways})$, $(\ref{inequn_strictsubsolutioncomparison})$ and the fact $\underline{u}\leq0$ on $\partial M$, we have
\begin{align}\label{inequn_strictsubsolutionbd2}
\sigma_k(\bar{\nabla}^2\underline{u})>\bar{\beta}_{k,n}e^{2k\underline{u}}
\end{align}
on $M$. Recall that $A>0$ and $p>0$ are independent of $x_0\in \partial M$. This proves the {\bf claim}. Therefore, we have constructed a strict sub-solution $\underline{u}\in C^{4,\alpha}(M)$ to $(\ref{equn_sigmak})$ and $\underline{u}\leq 0$ on $M$.

\section{A priori estimates for the $\sigma_k$-Ricci curvature flow}\label{section_C2estimates}

On a compact Riemannian manifold $(M^n, g)$ with boundary $\partial M$ of $C^{4,\alpha}$. We denote $M^{\circ}$ the interior of $M$. If $(M, g)$ is a bounded domain in the Euclidean space $\mathbb{R}^n$, we choose the natural extension $(M_1, g_1)$ which is a small tubular neighborhood of $M$ in $\mathbb{R}^n$, and the global subsolution used in \cite{GSW} has no singularity in $M$. For general compact Riemannian manifold $(M^n, g)$ with boundary, with the extension $(M_1, g_1)$ in Section \ref{section_sub1.1}, we choose $g_1$ (and hence $g$ on $M$) to be the conformal metric which has $-Ric_{g_1}\geq \delta_0g_1$ with $\delta_0>n-1$.

For $k=1,...,n$, we consider the Cauchy-Dirichlet problem of the $\sigma_k$-Ricci curvature flow
\begin{align}
&\label{equn_flowsigmak}2ku_t=\log(\sigma_k(\bar{\nabla}^2u))-\log(\bar{\beta}_{k,n})-2ku,\,\,\,\text{on}\,\,M\times [0,+\infty),\\
\label{equn_flowinitialdata}&u\big|_{t=0}=u_0,\\
&\label{equn_flowboundarydata}u\big|_{\partial M} =\phi,\,\,\,t\geq 0, 
\end{align}
where $u_0\in C^{4,\alpha}(M)$ is a subsolution to the $\sigma_k$-Ricci equation $(\ref{equn_sigmak})$, 
$\bar{\nabla}^2 u$ is defined in $(\ref{equn_thetensor})$, and $\phi\in C^{4+\alpha,2+\frac{\alpha}{2}}(\partial M\times[0,T])$ for all $T>0$, and moreover, $\phi$ satisfies $\phi_t\geq 0$ for $t\geq 0$, $\phi(t)\to +\infty$ as $t\to+\infty$. To guarantee that the solution $u$ to the Cauchy-Dirichlet problem of $(\ref{equn_flowsigmak})$ satisfies $u\in C^{4+\alpha,2+\frac{\alpha}{2}}(M\times[0,T_0))$ for some $T_0>0$, we need the compatible condition
\begin{equation}\label{equn_compatibleequations}
\left\{
\begin{aligned}
&u_0(x)=\phi(x,0),\,\,\text{for}\,\,\,x\in\partial M,\\
&2k\phi_t(x,0)=\log(\sigma_k(\bar{\nabla}^2u_0)(x))-\log(\bar{\beta}_{k,n})-2ku_0(x),\,\,\text{for}\,\,\,x\in\partial M,\\
&2k\phi_{tt}(x,0)=L_0(v(x)),\,\,\text{for}\,\,\,x\in\partial M,
\end{aligned}
\right.
\end{equation}
where the function $v\in C^2(M)$ is
\begin{align}\label{equn_defnfunctionA1}
v(x)\equiv\frac{1}{2k}(\log(\sigma_k(\bar{\nabla}^2u_0)(x))-\log(\bar{\beta}_{k,n})-2ku_0(x))
\end{align}
and $L_0$ is the linear operator
\begin{align*}
L_0(\varphi)=\frac{\bar{T}_{k-1}^{ij}}{\sigma_k(\bar{\nabla}^2u_0)}[(n-2)\nabla_i\nabla_j\varphi+\Delta \varphi g_{ij}+(n-2)(2g^{km}\nabla_ku_0\nabla_m\varphi g_{ij}-\nabla_i\varphi\nabla_ju_0-\nabla_iu_0\nabla_j\varphi)]-2k\varphi,
\end{align*}
for any $\varphi\in C^2(M)$, where $\bar{T}_{k-1}^{ij}$ is the $(k-1)$-th Newton transformation of $\bar{\nabla}^2u_0$, which is positive definite. In order to find boundary data $\phi\in C^{4+\alpha,2+\frac{\alpha}{2}}(\partial M\times[0,\infty))$ compatible with $u_0$ such that $\phi_t\geq0$ on $\partial M\times[0,\infty)$, we need to assume that
for the subsolution $u_0\in C^{4,\alpha}(M)$,
\begin{align}\label{inequn_c4compatiblebdincreasing}
L_0(v(x))\geq0
\end{align}
 at any point $x\in \partial M$ such that $v(x)=0$. We remark that sub-solutions $u_0$ to $(\ref{equn_sigmak})$ with the condition $(\ref{inequn_c4compatiblebdincreasing})$ always exist on $(M, g)$: It is clear that we do not need the condition $(\ref{inequn_c4compatiblebdincreasing})$ for a sub-solution $\underline{u}$ which is strict on $\partial M$ i.e.,
 \begin{align*}
 \sigma_k(\bar{\nabla}^2\underline{u})>\bar{\beta}_{k,n}e^{2k\underline{u}}
 \end{align*}
for all $x\in\partial M$. For instance, the global subsolution $\underline{u}$ we constructed in Section \ref{section_sub1.1}, by $(\ref{inequn_strictsubsolutionbd2})$. Another example is $u_0=\varphi-C$, with $\varphi$ a sub-solution of $(\ref{equn_sigmak})$ and $C>0$ a constant and hence, $u_0$ is a strict sub-solution of $(\ref{equn_sigmak})$ on $M$. Also, if $u_0\in C^{4,\alpha}(M)$ is a solution to $(\ref{equn_sigmak})$, then $v=0$ on $M$ and hence $(\ref{inequn_c4compatiblebdincreasing})$ holds automatically. When $u_0$ is a solution to $(\ref{equn_sigmak})$ with $u_0=0$ on $\partial M$ as obtained in \cite{GSW} and \cite{Guan1}, we can choose the boundary data $\phi=\phi(t)\in C^3$ such that $\phi(0)=\phi'(0)=\phi''(0)=0$ and $\phi'(t)\geq0$ for $t\geq0$. For a given constant $T>0$, we call a function $u\in C^2(M\times [0,T))$ a {\it sub-solution} ({\it super-solution}) of $(\ref{equn_flowsigmak})$ if $\bar{\nabla}^2u \in \Gamma_k^+$ and $u$ satisfies the inequality with "$\leq$" ("$\geq$") instead of "$=$" in $(\ref{equn_flowsigmak})$. Notice that sub-solution and super-solution are defined similarly for $(\ref{equn_sigmak})$.

We now prove a maximum principle, which serves as a comparison theorem for later use.
\begin{lem}\label{lem_comparison}
Let $u$ and $v$ be sub- and super- solutions to $(\ref{equn_flowsigmak})$, with $u\leq v$ on $\partial M\times [0,T)$ and $M\times\{0\}$, then we have $u\leq v$ on $M\times[0,T)$.
\end{lem}
\begin{proof}
The proof is a modification of the maximum principle of $\sigma_k$-Ricci equation in \cite{GSW}. We argue by contradiction. Let $\xi=u-v$. Assume that there exist $0<t_1<T$ and $x\in M^{\circ}$ such that
\begin{align*}
\xi(x,t_1)=\sup_{M\times[0,t_1]}\xi>0.
\end{align*}
Then we have at $(x, t_1)$,
\begin{align*}
&\tilde{u}_t\geq v_t,\,\,\,\nabla \tilde{u}=\nabla v,\\
&\nabla^2(v-\tilde{u})\geq 0,
\end{align*}
and hence
\begin{align*}
\bar{\nabla}^2\tilde{u}+ \mathcal{V}=\bar{\nabla}^2v
\end{align*}
with $\mathcal{V}=(n-2)\nabla^2(v-\tilde{u})+\Delta(v-\tilde{u})g\geq 0$, which implies that $\sigma_k(\bar{\nabla}^2\tilde{u})\leq \sigma_k(\bar{\nabla}^2v)$, and hence
\begin{align*}
2k\tilde{u}_t-\log(\sigma_k(\bar{\nabla}^2\tilde{u}))\geq 2kv_t-\log(\sigma_k(\bar{\nabla}^2v))
\end{align*}
at $(x,t_1)$. On the other hand, the function $\tilde{u}=u-\xi(x,t_1)$ is a strict sub-solution to $(\ref{equn_flowsigmak})$ on $M\times [0,T)$:
\begin{align*}
2k\tilde{u}_t=2ku_t\leq \log(\sigma_k(\bar{\nabla}^2u))-\log(\bar{\beta}_{k,n})-2ku<\log(\sigma_k(\bar{\nabla}^2\tilde{u}))-\log(\bar{\beta}_{k,n})-2k\tilde{u}.
\end{align*}
By the definition of sub- and super- solutions, we have at $(x,t_1)$,
\begin{align*}
2k\tilde{u}_t-\log(\sigma_k(\bar{\nabla}^2\tilde{u}))<-\log(\bar{\beta}_{k,n})-2k\tilde{u}=-\log(\bar{\beta}_{k,n})-2kv\leq 2kv_t-\log(\sigma_k(\bar{\nabla}^2v)),
\end{align*}
which is a contradiction. This proves the lemma.
\end{proof}

Based on the fact that the initial data $u_0$ is a subsolution of $(\ref{equn_sigmak})$ and the boundary data $\phi$ is increasing in $t$, we have the monotonicity lemma.

\begin{lem}\label{Lem_monotonicity}
Assume that $u_0\in C^3(M)$ is a subsolution to the $\sigma_k$-Ricci equation $(\ref{equn_sigmak})$, and $u\in C^{3,2}(M\times[0,T))$ is a solution to $(\ref{equn_flowsigmak})$ for some $T>0$. Assume that $u(x,t)=\phi(x,t)$ for any $(x,t)\in \partial M\times[0,T)$ and $\frac{\partial }{\partial t}\phi\geq 0$ on $\partial M\times[0,T)$. Then $u_t\geq0$ in $M\times [0,T)$. In particular, $u$ is increasing along $t\geq0$. Moreover, we have upper bound estimates for $u_t$ on $M\times[0,T)$.
\end{lem}
\begin{proof}
Let $v=u_t$. We take derivative of $t$ on both sides of the equation $(\ref{equn_flowsigmak})$ to have
\begin{align}\label{equn_flowlinear1}
2kv_t=\frac{1}{\sigma_k(\bar{\nabla}^2u)}\bar{T}_{k-1}^{ij}[(n-2)\nabla_i\nabla_jv+\Delta vg_{ij}+(n-2)(2g^{km}u_kv_mg_{ij}-v_iu_j-u_iv_j)]-2kv,
\end{align}
where $\bar{T}_{k-1}^{ij}$ is the $(k-1)$-th Newton transformation of $\bar{\nabla}^2u$, which is positive definite since $\bar{\nabla}^2u\in \Gamma_k^+$. Recall that $u_0$ is a subsolution of $(\ref{equn_sigmak})$, by the equation $(\ref{equn_flowsigmak})$ we have that $v(x,0)\geq 0$ for $x\in M$. Also, $v(x,t)=\phi_t(x,t)\geq 0$ for $(x,t)\in \partial M\times [0,T)$. We will use maximum principle to obtain that $v\geq 0$ on $M\times [0,T)$. Otherwise, assume that there exists $x_0\in M^{\circ}$ and $t_1\in(0,T)$ such that
\begin{align*}
v(x_0,t_1)=\inf_{M\times[0,t_1]}v<0,
\end{align*}
then at $(x_0,t_1)$, we have that
\begin{align*}
v_t\leq 0,\,\,\nabla v=0,\,\,\nabla^2 v\geq0,\,\,\,v<0,
\end{align*}
and hence
\begin{align*}
v_t\leq 0,\,\,\frac{1}{\sigma_k(\bar{\nabla}^2u)}\bar{T}_{k-1}^{ij}[(n-2)\nabla_i\nabla_jv+\Delta vg_{ij}+(n-2)(2g^{km}u_kv_mg_{ij}-v_iu_j-u_iv_j)]-2kv>0,
\end{align*}
at $(x_0,t_1)$, contradicting with the equation $(\ref{equn_flowlinear1})$. Therefore, $v=u_t\geq 0$ on $M\times[0,T)$. In particular, $u$ is a sub-solution to $(\ref{equn_sigmak})$ for each $t>0$.

Similarly, assume $v(x_0,t_1)=\sup_{M\times [0,t_1]}v>0$ for some $(x_0,t_1)\in M^{\circ}\times(0,T)$. Then at $(x_0,t_1)$,
\begin{align*}
v_t\geq 0,\,\,\frac{1}{\sigma_k(\bar{\nabla}^2u)}\bar{T}_{k-1}^{ij}[(n-2)\nabla_i\nabla_jv+\Delta vg_{ij}+(n-2)(2g^{km}u_kv_mg_{ij}-v_iu_j-u_iv_j)]-2kv<0,
\end{align*}
contradicting with the equation $(\ref{equn_flowlinear1})$. Therefore, combining with $(\ref{equn_flowsigmak})$ at $t=0$, we have
\begin{align*}
v(x,t)=u_t(x,t)\leq \max\{\frac{1}{2k}\sup_{M}[\log(\sigma_k(\bar{\nabla}^2u_0))-\log(\bar{\beta}_{k,n})-2ku_0],\,\,\sup_{\partial M\times[0,t]}\phi_t\}
\end{align*}
for any $(x,t)\in M\times [0,T)$. By integration, we have
\begin{align*}
u(x,t)&=u_0(x)+\int_0^tu_t(x,s)ds\\
&\leq u_0(x)+t\max\{\frac{1}{2k}\sup_{M}[\log(\sigma_k(\bar{\nabla}^2u_0))-\log(\bar{\beta}_{k,n})-2ku_0],\,\,\sup_{\partial M\times[0,t]}\phi_t\},
\end{align*}
for any $(x,t)\in M\times [0,T)$; on the other hand, by monotonicity, $u(x,t)\geq u_0(x)$. Hence, we obtain the upper and lower bound estimates for $u$ on $M\times [0,T)$.

\end{proof}

We then give the boundary $C^1$ estimates on $u$.

\begin{lem}\label{lem_gradientboundaryests}
Assume $(M^n, g)$ is  a compact manifold with boundary of $C^{4,\alpha}$, and $(M,g)$ is either a compact domain in $\mathbb{R}^n$ or with Ricci curvature $Ric_g\leq -\delta_0g$ for some $\delta_0\geq (n-1)$.  Let $u\in C^4(M\times[0,T_0))$ be a solution to the Cauchy-Dirichlet problem $(\ref{equn_flowsigmak})-(\ref{equn_flowboundarydata})$ for some $T_0>0$. Assume $u_0\in C^{4,\alpha}(M)$ is a subsolution to $(\ref{equn_sigmak})$ satisfying $(\ref{inequn_c4compatiblebdincreasing})$ at the points $x\in \partial M$ where $v(x)=0$. Also, assume $\phi\in C^{4+\alpha,2+\frac{\alpha}{2}}(\partial M\times[0,T_1])$ for all $T_1>0$, $\phi_t(x,t)\geq 0$ on $\partial M\times[0,+\infty)$ and $\phi$ satisfies the compatible condition $(\ref{equn_compatibleequations})$ with $u_0$. Then we have the boundary gradient estimates of $u$ i.e., there exists a constant $C=C(T_0)>0$ such that
\begin{align}
|\nabla u(x,t)|\leq C
\end{align}
for $(x,t)\in \partial M\times[0,T_0)$.
 \end{lem}

 \begin{proof}
 By the Dirichlet boundary condition, tangential derivatives of $u$ on $\partial M\times[0,t_0)$ is controlled by the tangential derivatives of the boundary data $\phi$ and hence,  for the boundary gradient estimates of $u$, we only need to control $|\frac{\partial}{\partial n}u|$ with $n$ the outer normal vector field of $\partial M$.

Since $\bar{\nabla}^2u\in \Gamma_k^+$, we will show the lower bound of $\frac{\partial}{\partial n}u$ based on the control of $\sup_{M\times [0,T_0)}|u|$ as Guan's argument in Lemma 5.2 in \cite{Guan1}. Indeed, we have
\begin{align*}
\text{tr}(\bar{\nabla}^2u)=2(n-1)[\Delta u +\frac{(n-2)}{2}|\nabla u|^2-\frac{1}{2(n-1)}R_g]\geq0,
\end{align*}
where $R_g\leq 0$ since $Ric_g\leq 0$. Let $v=e^{\frac{n-2}{2}u}$. Then we have
\begin{align*}
[\Delta v -\frac{n-2}{4(n-1)}R_gv]\geq0.
\end{align*}
Let $m=\sup_{M\times[0,T_0)}|u|$, which is bounded by the proof of Lemma \ref{Lem_monotonicity}. For any $t>0$, let $\tilde{v}=\tilde{v}(x,t)$ be the solution to the Dirichlet boundary value problem of the linear elliptic equation
\begin{align*}
&\Delta \tilde{v}=\frac{n-2}{4(n-1)}R_ge^{\frac{n-2}{2}m},\,\,\,\text{in}\,\,\,M,\\
&\tilde{v}(x,t)=e^{\frac{n-2}{2}\phi(x,t)},\,\,\,p\in\partial M.
\end{align*}
Then by continuity, for any $T>0$, there exists a uniform constant $C=C(T)>0$, such that
 \begin{align*}
 \sup_{(x,t)\in\partial M\times[0,T]}|\frac{\partial}{\partial n}\tilde{v}|\leq C(T)<+\infty.
 \end{align*}
For $t<T_0$, we have
\begin{align*}
&\Delta \tilde{v}(x,t)\leq\frac{n-2}{4(n-1)}R_gv(x,t)\leq \Delta v(x,t),\,\,\,\forall x\in M,\\
&\tilde{v}(x,t)=v(x,t),\,\,\,x\in\partial M.
\end{align*}
By maximum principle, $v(x,t)\leq \tilde{v}(x,t)$ in $M$ and since $v(x,t)=\tilde{v}(x,t)$ for $(x,t)\in\partial M\times[0,T_0)$, we have
\begin{align*}
\frac{\partial}{\partial n}v\geq\frac{\partial}{\partial n}\tilde{v}\geq -C
\end{align*}
for some uniform constant $C=C(T_0)>0$ on $\partial M\times [0,T_0)$, and hence
\begin{align*}
\frac{\partial}{\partial n}u\geq\frac{2}{n-2}e^{-\frac{n-2}{2}u}\frac{\partial}{\partial n}\tilde{v}\geq -\frac{2}{n-2}C(T_0)e^{-\frac{n-2}{2}\sup_{M\times[0,T_0)}|u|}
\end{align*}
for $(x,t)\in \partial M\times[0,T_0)$. This gives a uniform lower bound of $\frac{\partial}{\partial n}u$ on $\partial M\times[0,T_0)$.

Now we give upper bound estimates on $\frac{\partial}{\partial n}u$. Let $(M_1, g_1)$ be either a small tubular neighborhood of $(M,g)$ in $\mathbb{R}^n$, or an extension of $(M,g)$ as in Section \ref{section_sub1.1} respectively.
 For any $x_0\in \partial M$, let $\bar{x}\in M_1\setminus M$ be as in Section \ref{section_sub1.1} and $r(x)$ be the distance function to $\bar{x}$ in $M_1$ for $x\in M_1$. Let $\delta_1>0$ be a small constant such that $\delta_1<\delta$ with $\delta>0$ defined in Section \ref{section_sub1.1}. Define the domain $U=\{x\in M,\,r(x)\leq r(x_0)+\delta_1\}$, with its boundary $\partial U=\Gamma_0\bigcup \Gamma_1$ where $\Gamma_0= U \bigcap \partial M$ and $\Gamma_1=\{x\in M\big|\,r(x)=r(x_0)+\delta_1\}$. Since $2\delta+r(x_0)$ is less than the injectivity radius at $\bar{x}$, $r(x)$ is smooth in $U$.
 For given $T>0$, we extend $\phi$ to a $C^{4+\alpha,2+\frac{\alpha}{2}}$
function on $U\times[0,T]$ for any $T > 0$ so that $\phi(x,0)=u_0(x)$ for $x\in U$. Define the function
\begin{align*}
\underline{u}(x,t)=\phi(x,t)+A(\frac{1}{r(x)^p}-\frac{1}{r(x_0)^p}),
\end{align*}
on $U\times [0,T]$, with two large constants $A>0$ and $p>0$ to be determined. We will choose $A=A(T)$ and $p=p(T)$ large so that $\underline{u}$ is a barrier function that controls the lower bound of $u$ on $U\times [0,T]$. Direct computations lead to
\begin{align*}
&\underline{u}_t=\phi_t,\\
&\nabla \underline{u}= \nabla \phi-Ap\,r^{-p-1}\nabla r,\\
&\nabla_i \nabla_j \underline{u}=\nabla_i\nabla_j\phi+Ap(p+1)\,r^{-p-2}\nabla_i r \nabla_j r\,-\,Ap\,r^{-p-1}\nabla_i \nabla_j r\\
&\Delta \underline{u}= \Delta\phi+Ap(p+1)\,r^{-p-2}\big|\nabla r\big|^2-Ap\,r^{-p-1}\Delta r=\Delta\phi+Ap(p+1)\,r^{-p-2}-Ap\,r^{-p-1}\Delta r.
\end{align*}
By continuity, there exist constants $C_1>0$ and $C_2=C_2(T)>0$ such that $|\nabla^2r|+|\Delta r|\leq C_1$ in $U$ and $|\nabla \phi|+|\nabla^2\phi|+|\Delta\phi|\leq C_2$ in $U\times [0,T]$. We have the calculation
\begin{align*}
(\bar{\nabla}^2\underline{u})_{ij}&=-Ric_{ij}(g)+(n-2)[\nabla_i\nabla_j\phi+Ap(p+1)\,r^{-p-2}\nabla_i r \nabla_j r\,-\,Ap\,r^{-p-1}\nabla_i \nabla_j r]\\
&\,\,\,\,\,\,+[\Delta\phi+Ap(p+1)\,r^{-p-2}-Ap\,r^{-p-1}\Delta r]g_{ij}+(n-2)[|\nabla\underline{u}|^2g_{ij}-\nabla_i\underline{u}\nabla_j\underline{u}].
\end{align*}
Since $-Ric_g\geq0$ and the matrix $(\nabla_ir\nabla_jr)$ and the last term are semi-positive, we have
\begin{align*}
(\bar{\nabla}^2\underline{u})_{ij}\geq (n-2)[\nabla_i\nabla_j\phi\,-\,Ap\,r^{-p-1}\nabla_i \nabla_j r]+[\Delta\phi-Ap\,r^{-p-1}\Delta r]g_{ij}+Ap(p+1)\,r^{-p-2}g_{ij},
\end{align*}
and hence for any $N_1>0$ and $A>0$, there exists a constant $p_0=p_0(T,N_1,A)>0$, such that for $p>p_0$, we have
\begin{align*}
(\bar{\nabla}^2\underline{u})_{ij}\geq N_1g_{ij}
\end{align*}
on $U\times [0,T]$. Let
\begin{align*}
N_1\geq\,\bar{\beta}_{n,n}^{\frac{1}{n}} e^{2\sup_{U\times [0,T]}|\phi_t|+2\sup_{U\times [0,T]}|\phi|}.
\end{align*}
Then we have
\begin{align*}
\log(\sigma_n(\bar{\nabla}^2\underline{u}))&\geq \log(N_1^n)\geq 2n\phi_t+\log(\bar{\beta}_{n,n})+2n\phi\\
&\geq 2n\underline{u}_t+\log(\bar{\beta}_{n,n})+2n\underline{u}
\end{align*}
on $U\times [0,T]$. Therefore, $\underline{u}$ is a subsolution of the $\sigma_n$-Ricci curvature flow. By Maclaurin's inequality, $\underline{u}$ is a subsolution of the $\sigma_k$-Ricci curvature flow for any $1\leq k\leq n$. By definition, we know that $\underline{u}\leq u$ on $\Gamma_0\times[0,T_0)$. On $\Gamma_1\times[0,T_0)$, $u$ and $\phi$ has uniform upper and lower bounds, and hence we can choose $A$ and $p$ large enough so that $\underline{u}<u$ on $\Gamma_1\times[0,T_0)$. Also, we have
\begin{align*}
\underline{u}(x,0)\leq \phi(x,0)=u_0(x)
\end{align*}
for $x\in U$. By maximum principle in Lemma \ref{lem_comparison}, we have that
\begin{align*}
u\geq \underline{u}
\end{align*}
in $U\times [0,T_0)$. Since $u(x_0,t)=\phi(x_0,t)=\underline{u}(x_0,t)$, we have
\begin{align*}
\frac{\partial}{\partial n}u\leq \frac{\partial}{\partial n}\underline{u}
\end{align*}
at $(x_0,t)$ for $t\in[0, T_0]$, where $n$ is the unit outer normal vector of $\partial M$ at $x_0$. Notice that the constants used here can be chosen uniformly for all $x_0\in \partial M$ and hence, there exists a unique constant $m_1=m_1(T_0)>0$, such that $\frac{\partial}{\partial n}u\leq m_1$ on $\partial M\times[0,T_0)$. Therefore, we have the $C^1$ estimates of $u$ at points on $\partial M$ i.e., there exists a constant $C=C(T_0)>0$ such that
\begin{align*}
|\nabla u(x,t)|\leq C
\end{align*}
for $(x,t)\in \partial M\times[0,T_0)$.
\end{proof}
Now we give the $C^1$ estimates of $u$ on $M\times[0,T_0)$.

\begin{lem}\label{lem_c1estimatesonT}
Let $(M,g)$ and $u\in C^4(M\times[0,T_0))$ be as in Lemma \ref{lem_gradientboundaryests}. Then there exists a constant $C=C(T_0)>0$ such that
\begin{align*}
|\nabla u(x,t)|\leq C
\end{align*}
for $(x,t)\in M\times [0,T_0)$.
\end{lem}
\begin{proof}
The interior gradient estimate is relatively standard, and here we modify the argument in \cite{LS} (see also \cite{GV}). By Lemma \ref{Lem_monotonicity}, there exist two constants $-\infty<\beta_1<\beta_2<+\infty$ depending on $T_0$ such that $\beta_1\leq u\leq \beta_2$ on $M\times[0,T_0)$. We define a function
\begin{align*}
\xi(x,t)=(1+\frac{|\nabla u|^2}{2})e^{\eta(u)},
\end{align*}
where
\begin{align*}
\eta(s)=C_1(C_2+s)^p
 \end{align*}
 is a function on $s\in[\beta_1,+\infty)$ with constants $C_2>-\beta_1$, $C_1>0$ and $p>0$, depending only on $T_0,\,\beta_1$ and $\beta_2$, to be determined. Suppose that there exists $x_0\in M^{\circ}$ and $t_0\in(0,T_0)$ such that
 \begin{align*}
\xi(x_0,t_0)=\sup_{M\times[0,t_0]}\xi.
\end{align*}
We take geodesic normal coordinates $(x^1,...,x^n)$ centered at $x_0\in M$ such that $\Gamma_{ij}^m(x_0)=0$, $g_{ij}(x_0)=\delta_{ij}$. Then we have at $(x_0,t_0)$,
\begin{align}
\xi_{x_i}&\label{equn_interiord1}=e^{\eta(u)}[u_{x_ax_i}u_{x_a}+(1+\frac{1}{2}u_{x_a}u_{x_a})\eta'(u)u_{x_i}]=0,\\
\xi_t&\label{inequn_maximumpointtimed1}=e^{\eta(u)}[u_{x_at}u_{x_a}+(1+\frac{u_{x_a}u_{x_a}}{2})\eta'(u)u_t]\geq0,\\
0\geq \xi_{x_ix_j}&=[\frac{1}{2}\frac{\partial^2}{\partial x_ix_j}g^{ab}u_{x_a}u_{x_b}+u_{x_ax_ix_j}u_{x_a}+u_{x_ax_i}u_{x_ax_j}+\eta'(u)u_{x_ax_i}u_{x_a}u_{x_j}+\eta'(u)u_{x_ax_j}u_{x_a}u_{x_i}\nonumber\\
&+(1+\frac{1}{2}|\nabla u|^2)(\eta'(u))^2u_{x_i}u_{x_j}+(1+\frac{1}{2}|\nabla u|^2)\eta''(u)u_{x_i}u_{x_j}+(1+\frac{1}{2}|\nabla u|^2)\eta'(u)u_{x_ix_j}]e^{\eta(u)}\nonumber\\
&=[\frac{1}{2}\frac{\partial^2}{\partial x_ix_j}g^{ab}u_{x_a}u_{x_b}+u_{x_ax_ix_j}u_{x_a}+u_{x_ax_i}u_{x_ax_j}+(1+\frac{1}{2}|\nabla u|^2)(\eta''(u)-(\eta'(u))^2)u_{x_i}u_{x_j}\nonumber\\
&+(1+\frac{1}{2}|\nabla u|^2)\eta'(u)u_{x_ix_j}]e^{\eta(u)},\nonumber
\end{align}
where the last identity is by $(\ref{equn_interiord1})$. Notice that the tensor
\begin{align*}
\bar{Q}_{ij}\equiv \frac{1}{\sigma_k(\bar{\nabla}^2u)}((n-2)(\bar{T}_{k-1})_{ij}+g^{ab}(\bar{T}_{k-1})_{ab}g_{ij}),
\end{align*}
is positive definite. Therefore, at $(x_0,t_0)$,
\begin{align}
0\,\,\,\,\,\,\,\,\,\geq\,\,&\nonumber[\frac{1}{(1+\frac{1}{2}|\nabla u|^2)}(\bar{Q}_{ij}u_{x_ix_j x_a}u_{x_a}+\frac{1}{2}\bar{Q}_{ij}\frac{\partial^2}{\partial x_i \partial x_j}g^{ab}u_{x_a}u_{x_b}+\bar{Q}_{ij}u_{x_ax_i}u_{x_ax_j})\\
&\label{inequn_maximumd2}+(\eta''(u)-(\eta'(u))^2)\bar{Q}_{ij}u_{x_i}u_{x_j}+\eta'(u)\bar{Q}_{ij}u_{x_ix_j}]e^{\eta(u)}.
\end{align}
By definition, at $(x_0,t_0)$ we have
\begin{align*}
\bar{\nabla}^2u=-Ric_g+(n-2)u_{x_ix_j}+\Delta u \delta_{ij}-(n-2)u_{x_i}u_{x_j}+(n-2)|\nabla u|^2\delta_{ij},
\end{align*}
and hence by the identity $\bar{T}_{ij}(\bar{\nabla}^2u)_{ij}=k\sigma_k(\bar{\nabla}^2u)$ and the equation $(\ref{equn_flowsigmak})$, we obtain
\begin{align}
\bar{Q}_{ij}u_{x_ix_j}&=\frac{1}{\sigma_k(\bar{\nabla}^2u)}[\bar{T}_{ij}(\bar{\nabla}^2u)_{ij}+\bar{T}_{ij}\big(Ric_{ij}+(n-2)u_{x_i}u_{x_j}-(n-2)|\nabla u|^2\delta_{ij}\big)]\nonumber\\
&\label{equn_flowformchanging}=\frac{1}{\sigma_k(\bar{\nabla}^2u)}[k\bar{\beta}_{k,n}e^{2ku_t+2ku}+\bar{T}_{ij}\big(Ric_{ij}+(n-2)u_{x_i}u_{x_j}-(n-2)|\nabla u|^2\delta_{ij}\big)],
\end{align}
 at $(x_0,t_0)$. Now take derivative of $x_i$ on both sides of $(\ref{equn_flowsigmak})$, and we have at $(x_0,t_0)$,
\begin{align*}
2ku_{tx_i}&=\frac{1}{\sigma_k(\bar{\nabla}^2u)}\bar{T}_{ab}[-\frac{\partial}{\partial x_i}Ric_{ab}+(n-2)u_{x_ax_bx_i}-(n-2)\frac{\partial}{\partial x_i}\Gamma_{ab}^mu_{x_m}\\
&+(u_{x_mx_mx_i}-\frac{\partial}{\partial x_i}\Gamma_{mm}^cu_{x_c})g_{ab}+(n-2)(2u_{x_mx_i}u_{x_m}g_{ab}-u_{x_ax_i}u_{x_b}-u_{x_a}u_{x_bx_i})]-2ku_{x_i},
\end{align*}
and hence at $(x_0,t_0)$, for $1\leq a\leq n$,
\begin{align*}
\bar{Q}_{ij}u_{x_ix_jx_a}&=2k(u_{tx_a}+u_{x_a})+\frac{1}{\sigma_k(\bar{\nabla}^2u)}\bar{T}_{ij}[(n-2)(-2u_{x_mx_a}u_{x_m}g_{ij}+u_{x_ix_a}u_{x_j}+u_{x_i}u_{x_jx_a}+\frac{\partial}{\partial x_a}\Gamma_{ij}^mu_{x_m})\\
&+\frac{\partial}{\partial x_a}\Gamma_{mm}^cu_{x_c}g_{ij}+\frac{\partial}{\partial x_a}Ric_{ij}].
\end{align*}
Now contracting this equation with $\nabla u$ we have at $(x_0,t_0)$,
\begin{align}
\bar{Q}_{ij}u_{x_ix_jx_a}u_{x_a}&=2k(u_{tx_a}u_{x_a}+u_{x_a}u_{x_a})+\frac{1}{\sigma_k(\bar{\nabla}^2u)}\bar{T}_{ij}[(n-2)(-2u_{x_mx_a}u_{x_m}u_{x_a}g_{ij}+2u_{x_ix_a}u_{x_j}u_{x_a}\nonumber\\
&+\frac{\partial}{\partial x_a}\Gamma_{ij}^mu_{x_m}u_{x_a})+\frac{\partial}{\partial x_a}\Gamma_{mm}^cu_{x_c}u_{x_a}g_{ij}+u_{x_a}\frac{\partial}{\partial x_a}Ric_{ij}]\nonumber\\
&\geq \frac{1}{\sigma_k(\bar{\nabla}^2u)}\bar{T}_{ij}[(n-2)\big(2(1+\frac{1}{2}|\nabla u|^2)\eta'(u)(u_{x_a}u_{x_a}g_{ij}
-u_{x_i}u_{x_j})
+\frac{\partial}{\partial x_a}\Gamma_{ij}^mu_{x_m}u_{x_a}\big) \nonumber\\
&\label{equn_contractiond1}+\frac{\partial}{\partial x_a}\Gamma_{mm}^cu_{x_c}u_{x_a}g_{ij}+u_{x_a}\frac{\partial}{\partial x_a}Ric_{ij}] +2k(u_{x_a}u_{x_a}-(1+\frac{|\nabla u|^2}{2})\eta'(u)u_t),
\end{align}
where the last inequality is by $(\ref{equn_interiord1})$ and $(\ref{inequn_maximumpointtimed1})$.
Substituting $(\ref{equn_flowformchanging})$ and $(\ref{equn_contractiond1})$ to $(\ref{inequn_maximumd2})$, we have
\begin{align*}
0\geq &\frac{1}{(1+\frac{1}{2}|\nabla u|^2)}[2k(|\nabla u|^2-(1+\frac{|\nabla u|^2}{2})\eta'(u)u_t)+\frac{\bar{T}_{ij}}{\sigma_k(\bar{\nabla}^2u)}\frac{\partial}{\partial x_a}Ric_{ij}u_{x_a}+\bar{Q}_{ij}u_{x_ax_i}u_{x_ax_j}+\bar{Q}_{ij}R_{iajb}u_{x_a}u_{x_b}]\\
&+\frac{2(n-2)}{\sigma_k(\bar{\nabla}^2u)}\bar{T}_{ij}\,(|\nabla u|^2g_{ij}
-u_{x_i}u_{x_j})\eta'(u)+(\eta''-(\eta')^2)\bar{Q}_{ij}u_{x_i}u_{x_j}\\
&+\,\frac{\eta'}{\sigma_k(\bar{\nabla}^2u)}[k\bar{\beta}_{k,n}e^{2ku_t+2ku}+\bar{T}_{ij}\,(Ric_{ij}+(n-2)(u_{x_i}u_{x_j}-|\nabla u|^2g_{ij}))]\\
=&(n-2)\frac{1}{\sigma_k(\bar{\nabla}^2u)}(\eta''-(\eta')^2-\eta')\bar{T}_{ij}u_{x_i}u_{x_j}+\frac{1}{\sigma_k(\bar{\nabla}^2u)}(\eta''-(\eta')^2+(n-2)\eta')|\nabla u|^2\sum_i\bar{T}_{ii}\\
&+\eta'\,\frac{k\bar{\beta}_{k,n}e^{2ku+2ku_t}}{\sigma_k(\bar{\nabla}^2u)}+\,\frac{\eta'}{\sigma_k(\bar{\nabla}^2u)}\bar{T}_{ij}Ric_{ij}-2k\eta'\,u_t\\
&+\frac{1}{(1+\frac{1}{2}|\nabla u|^2)}[2k|\nabla u|^2+\frac{\bar{T}_{ij}}{\sigma_k(\bar{\nabla}^2u)}\frac{\partial}{\partial x_a}Ric_{ij}u_{x_a}+\bar{Q}_{ij}u_{x_ax_i}u_{x_ax_j}+\bar{Q}_{ij}R_{iajb}u_{x_a}u_{x_b}].
\end{align*}
Recall that $u$ and $u_t$ are uniformly bounded from above and blow on $M\times[0,T_0)$ by Lemma \ref{Lem_monotonicity}, and so is the term
\begin{align*}
\frac{1}{\sigma_k(\bar{\nabla}^2u)}=\bar{\beta}_{k,n}^{-1}e^{-2ku_t-2ku}.
 \end{align*}
Since $\bar{T}_{k-1}$ and $\bar{Q}_{k-1}$ are positively definite, we have at $(x_0,t_0)$,
\begin{align*}
0\geq &(n-2)\frac{e^{2ku}}{\sigma_k(\bar{\nabla}^2u)}(\eta''-(\eta')^2-\eta')\bar{T}_{ij}u_{x_i}u_{x_j}+\frac{e^{2ku}}{\sigma_k(\bar{\nabla}^2u)}(\eta''-(\eta')^2+(n-2)\eta')|\nabla u|^2\sum_{i}\bar{T}_{ii}\\
&-C-C\sum_{i}\bar{T}_{ii},
 \end{align*}
with the constant $C>0$ depending on $T_0$, $\sup_{\partial M\times[0,T_0)}(|\phi|+|\phi_t|)$, $\sup_{M}\log(\sigma_k(\bar{\nabla}^2u_0))$, $\sup_{M}|u_0|$, $\sup_{M}(|Ric_g|+|\nabla Ric_g|)$ and $\sup_{\beta_1\leq s\leq \beta_2}|\eta'(s)|$. By the definition of $\eta$, we have $\eta'>0$, and
\begin{align*}
\eta''-(\eta')^2-\eta'=C_1p(C_2+s)^{p-2}[(p-1)-C_1p(C_2+s)^p-(C_2+s)].
\end{align*}
For $\beta_1\leq s\leq \beta_2$, we choose $C_2=1-\beta_1$, $p>0$ large and then choose $C_1>0$ small so that
\begin{align*}
&\eta''-(\eta')^2\geq C_1p,\\
&\eta''-(\eta')^2-\eta'\geq0,
\end{align*}
and hence at $(x_0,t_0)$
\begin{align*}
 |\nabla u|^2\sum_i\bar{T}_{ii}\leq \frac{\sigma_k(\bar{\nabla}^2u)}{C_1pe^{2ku}}(C+C\sum_{i}\bar{T}_{ii})=\frac{1}{C_1p}\bar{\beta}_{k,n}e^{2ku_t}(C+C\sum_{i}\bar{T}_{ii})\leq \bar{C}(1+\sum_{i}\bar{T}_{ii}),
\end{align*}
where the constant $\bar{C}>0$ depends on $T_0$, $\sup_{\partial M\times[0,T_0)}(|\phi|+|\phi'|)$, $\sup_{M}\log(\sigma_k(\bar{\nabla}^2u_0))$, $\sup_{M}|u_0|$, $\sup_{M}(|Ric_g|+|\nabla Ric_g|)$ and $\sup_{\beta_1\leq s\leq \beta_2}|\eta'(s)|$. Recall that
\begin{align}
\sum_{i}\bar{T}_{ii}=(n-k+1)\,\sigma_{k-1}(\bar{\nabla}^2u)&\geq (n-k+1)\left(\begin{matrix} n\\ k-1\end{matrix}\right)\,\,\big(\left(\begin{matrix} n\\ k\end{matrix}\right)^{-1}\sigma_k(\bar{\nabla}^2u)\,\big)^{\frac{k-1}{k}}\nonumber\\
&\label{inequn_structurebounda1}=(n-k+1)\left(\begin{matrix} n\\ k-1\end{matrix}\right)\big(\left(\begin{matrix} n\\ k\end{matrix}\right)^{-1}\bar{\beta}_{k,n}e^{2ku_t+2ku}\big)^{\frac{k-1}{k}}\geq C,
\end{align}
for some uniform constant $C=C(T_0)>0$, where we have used the Maclaurin's inequality and the uniform lower bound of $u$ and $u_t\geq0$. Therefore, 
\begin{align*}
|\nabla u|^2\leq \bar{C}(1+\frac{1}{C}).
\end{align*}
This combining with the boundary $C^1$ estimates completes the proof of the gradient estimates of $u$ on $M\times[0,T_0)$.

\end{proof}

Now we consider the $C^2$ estimates on $u$ at the points on $\partial M\times[0,T_0)$.

 \begin{lem}\label{lem_c2boundary1-1}
 Let $(M,g)$ and $u\in C^4(M\times[0,T_0))$ be as in Lemma \ref{lem_gradientboundaryests}. Then there exists a constant $C=C(T_0)>0$ such that
 \begin{align*}
 |\nabla^2u|\leq C
 \end{align*}
 on $\partial M\times[0,T_0)$.
 \end{lem}
 \begin{proof}
 We use the indices $e_i,\,e_j$ to refer to the tangential vector fields on $\partial M$ and $n$ the outer normal vector field. Notice that we have obtained the uniform bounds
\begin{align*}
\sup_{\partial M\times[0,T_0)}(|u|+|\nabla u|)\leq K,
\end{align*}
for some constant $K>0$ on $\partial M\times[0,T_0)$. By definition, we immediately have the control on the second order tangential derivatives
\begin{align*}
\sup_{\partial M\times[0,T_0)}|\nabla_i\nabla_ju|\leq C
\end{align*}
on $\partial M\times[0,T_0)$ with some constant $C>0$ depending on $K$ and $\sup_{\partial M\times[0,t_0]}(|\phi|+|\nabla\phi|+|\nabla_{\tau}^2\phi|)$ where $\nabla_{\tau}^2\phi$ means the second order tangential derivatives of $\phi$ on $\partial M$. We extend $\phi$ to a function in $C^{4,2}(U\times[0,+\infty))$ still denoted as $\phi$ such that $\phi\in C^{4+\alpha,2+\frac{\alpha}{2}}(M\times[0,T])$ for any $T>0$ and $\phi(x,0)=u_0(x)$ for $x\in M$.

We now estimate the mixed second order derivatives $|\nabla_n\nabla_iu|$ with $n$ the normal vector field on $\partial M$. Let $(M_1,g_1)$ be the extension of $(M,g)$ as in Section \ref{section_sub1.1}.
Let $\delta>\epsilon_1>0$ be the small constants in Section \ref{section_sub1.1}. For any $x_0\in \partial M$, let $\bar{x}$ be the point with respect to $x_0$ as defined in Section \ref{section_sub1.1}.
 Define the exponential map $\text{Exp}: \partial M \times[-\epsilon_1-2\delta,\epsilon_1+2\delta]\to M_1$ such that $\text{Exp}_{q}(s)$ is the point along the geodesic starting from $q\in \partial M$ in the normal direction of $\partial M$ of distance $|s|$ to $q$.
 Here we take the inner direction to be positive i.e., $\text{Exp}_{q}(s)\in M^{\circ}$ when $s>0$. In particular, $\bar{x}=\text{Exp}_{x_0}(-\epsilon_1)$. Notice that $\text{Exp}: \partial M \times[-\epsilon_1-2\delta,\epsilon_1+2\delta]$ is a diffeomorphism to its image. In fact we can choose $\epsilon_1+2\delta< \epsilon$ where $\epsilon$ is strictly less than the lower bound of injectivity radius of each point in the thin $(\epsilon_1+2\delta)$-neighborhood $\Omega$ of $\partial M$. We now use the Femi coordinate in a small neighborhood $V_{x_0}=B_{\epsilon}(x_0)$ of $x_0$ in $M_1$: Let $(x^1,...,x^{n-1})$ be a geodesic normal coordinate centered at $x_0$ on $(\partial M, g\big|_{\partial M})$. We take $(x^1(q),...,x^{n-1}(q), x^n)$ as the coordinate of the point $\text{Exp}_q(x^n)$ in $V_{x_0}$. Define the distance function $r(x)=\text{dist}(x,\bar{x})$ for $x\in M_1$. Denote $U=\{x\in M\big|\,r(x)\leq\delta+r(x_0)\}$, $\Gamma_0=U\bigcap \partial M$ and $\Gamma_1=\{x\in M\big|\,r(x)=\delta+r(x_0)\}$. By our choice of the small constant $\epsilon_1+2\delta$, we have $\Gamma_0\subseteq V_{x_0}$ and hence $\frac{\partial}{\partial x^i}$ ($i<n$) is a tangential derivative of $\partial M$ on $\Gamma_0$. It is clear that $r(x)$ is smooth on $U$. The metric has the orthogonal decomposition
\begin{align*}
g=d(x^n)^2+g_{x^n}
\end{align*}
in $U$ and we have $\Gamma_{ab}^c(x_0)=0$ for $a,b,c\in\{1,2,...,n\}$. For $i\in\{1,...,n-1\}$, taking derivative of $\frac{\partial}{\partial x^i}$ on both sides of $(\ref{equn_flowsigmak})$ we have
\begin{align}
0=&-2ku_{tx_i}-2ku_{x_i}+\frac{1}{\sigma_k(\bar{\nabla}^2u)}\bar{T}_{ab}[-\nabla_iRic_{ab}+(n-2)\nabla_i\nabla_a\nabla_bu+\nabla_i\Delta ug_{ab}\nonumber\\
&\label{equn_floweqnd1}+2(n-2)\big(\nabla_i\nabla_c u\nabla_c ug_{ab}-\nabla_i\nabla_au\nabla_bu\big)].
\end{align}
Now we commute derivatives to have
\begin{align*}
&\nabla_i\nabla_a\nabla_bu=\nabla_a\nabla_bu_{x_i}+\text{Rm}\ast\nabla u,\\
&\nabla_i\Delta u= \Delta u_{x_i} +\text{Rm}\ast \nabla u,
\end{align*}
where the terms $\text{Rm}\ast \nabla u$ are contractions of some Riemannian curvature terms and $\nabla u$. Define the linearized operator $L$ acting on $\varphi$ as
\begin{align}\label{equn_linearoperator1}
L(\varphi)\equiv&\frac{1}{\sigma_k(\bar{\nabla}^2u)}\bar{T}_{ab}[(n-2)\nabla_a\nabla_b\varphi+\Delta \varphi g_{ab}+2(n-2)\big(<\nabla \varphi,\nabla u>g_{ab}-\nabla_a\varphi\nabla_bu\big)]\\
&-2k\varphi_t-2k\varphi.\nonumber
\end{align}
Therefore, by $(\ref{equn_floweqnd1})$ we have
\begin{align}
|L(u_{x_i})|= \frac{1}{\sigma_k(\bar{\nabla}^2u)}|\bar{T}_{ab}(-\nabla_iRic_{ab}+(Rm\ast \nabla u))|&\leq C\sum_i\bar{T}_{ii}(1+|\nabla u|)\nonumber\\
&\label{inequn_innerlinear}\leq C\sum_{i}\bar{T}_{ii},
\end{align}
for some constant $C>0$ depending on $\sup_{M} |Rm|$, the lower bound of $u_t+u$ and the upper bound of $|\nabla u|$ on $M\times[0,T_0)$, which has been uniformly controlled.
Recall that by $(\ref{inequn_structurebounda1})$, we have
\begin{align*}
\sum_{i}\bar{T}_{ii}\geq C
\end{align*}
for some uniform constant $C=C(T_0)>0$, and hence direct calculation leads to the bound
\begin{align}\label{inequn_bdlinear}
|L(\phi_{x_i})|\leq C\sum_i \bar{T}_{ii}+C \leq C\sum_i \bar{T}_{ii},
\end{align}
on $U\times[0,T_0)$, where $C>0$ in the inequalities are uniform constants depending on $T_0$, $k$, $n$, $\sup_{M\times[0,T_0)}(|u|+|u_t|+|\nabla u|)$ and $\sup_{U\times[0,T_0]}(|\phi_{x_i}|+|\phi_{tx_i}|+|\nabla\phi_{x_i}|+|\nabla^2\phi_{x_i}|)$. Define the function $v=u_{x_i}-\phi_{x_i}$ in $U\times[0,T_0)$. Now by $(\ref{inequn_innerlinear})$ and $(\ref{inequn_bdlinear})$ we have
\begin{align*}
|L(v)|\leq C\sum_i\bar{T}_{ii},
\end{align*}
for some uniform constant $C=C(T_0)>0$. Also, $v=0$ on $\Gamma_0$.

Now let
\begin{align*}
\xi(x)=\frac{1}{r(x)^p}-\frac{1}{r(x_0)^p}
\end{align*}
for $x\in U$, where $p>0$ is a constant depending on $T_0$ to be determined. Following the calculation in Section \ref{section_sub1.1}, we have that for $p=p(T_0)>0$ large,
\begin{align*}
(n-2)\nabla^2\xi+\Delta \xi g\geq \frac{p^2}{4}r^{-p-2}g.
\end{align*}
Since $\xi\leq0$, $|\nabla u|$ is uniformly bounded from above and $u_t+u$  is uniformly bounded from blow, we choose $p=p(T_0)>0$ large so that
\begin{align*}
L(\xi)&\geq \frac{1}{\bar{\beta}_{k,n}e^{2ku_t+2ku}}[\frac{p^2}{4}r^{-p-2}-C|\nabla u|\,|\nabla \xi|]\sum_i\bar{T}_{ii}-2k\xi\\
&\geq \frac{1}{C}(\frac{p^2}{4}r^{-p-2}-Cpr^{-p-1})\sum_i\bar{T}_{ii}\geq \frac{p^2}{8C}r^{-p-2}\sum_i\bar{T}_{ii}\\
&\geq |L(v)|
\end{align*}
on $U\times[0,T_0)$ for some uniform constant $C=C(T_0)>0$. Now we take $p>0$ even larger so that $\xi<-|v|$ on $\Gamma_1\times[0,T_0)$ and hence, $\xi\leq -|v|$ on $\partial U\times[0,T_0)$. Recall that
\begin{align*}
\xi(x)\leq0=v(x,0)
\end{align*}
for $x\in M$, we have by maximum principle,
\begin{align*}
\pm v(x,t)\geq \xi(x)
\end{align*}
for $(x,t)\in U\times[0,T_0)$. Since $v(x_0,t)=\xi(x_0)=0$, we have for $i=1,...,n-1$,
\begin{align*}
|\nabla_nu_{x_i}(x_0,t)|\leq |\nabla_n\phi_{x_i}(x_0,t)|+|\nabla_nv_{x_i}(x_0,t)|\leq |\nabla_n\phi_{x_i}(x_0,t)|+\nabla_n\xi(x_0)\leq C,
\end{align*}
for any $(x_0,t)\in\partial M\times[0,T_0)$ with some uniform constant $C=C(T_0)>0$ independent of the choice of $(x_0,t)\in \partial M\times[0,T_0)$, where $\nabla_n$ is the outer normal derivative at $x_0\in \partial M$. For the second order normal derivative $\nabla_n^2u$, since $\text{tr}(\bar{\nabla}^2u)\geq0$, i.e.
\begin{align*}
2(n-1)\Delta u+(n-2)(n-1)|\nabla u|^2-R_g\geq0,
\end{align*}
by the estimates on the other second order derivatives, $\nabla_{n}^2u$ is bounded from below and we still need to derive an upper bound of $\nabla_{n}^2u$. Orthogonally decompose the matrix $\bar{\nabla}^2u$ at $x_0\in \partial M$ in normal and tangential directions. By the previous estimates we have
\begin{align*}
\bar{\nabla}^2u=\left(\begin{matrix}(n-1)u_{nn}&0\\ 0&u_{nn}g\big|_{\partial M}\end{matrix}\right)+O(1)
\end{align*}
with the term $|O(1)|\leq C$ for some uniform constant $C=C(T_0)>0$ and hence, as the term $u_{nn}\to +\infty$, we have
\begin{align*}
\sigma_k(\bar{\nabla}^2u)=(u_{nn})^k\,(\Lambda_{k,n}+o(1))\to +\infty,
\end{align*}
where $\Lambda_{k,n}$ is a positive constant. On the other hand, recall that
\begin{align*}
0<\frac{1}{C}\leq\sigma_k(\bar{\nabla}^2u)=\bar{\beta}_{k,n}e^{2ku_t+2ku}\leq C,
\end{align*}
for some uniform constant $C=C(T_0)>0$ on $M\times[0,T_0)$ and hence, we have that there exists a uniform constant $C=C(T_0)>0$ such that $\nabla_n^2u(x_0)\leq C$. Notice that the constant $C$ here is independent of the choice of $x_0\in \partial M$. This completes the boundary $C^2$ estimates of $u$.

\end{proof}

\begin{prop}\label{prop_global apestimatesd2}
Let $(M,g)$ and $u\in C^4(M\times[0,T_0))$ be as in Lemma \ref{lem_gradientboundaryests}. Then there exists a constant $C=C(T_0)>0$ such that for any $(x,t)\in M\times[0,T_0)$ we have
\begin{align*}
|\nabla^2u(x,t)|_g\leq C.
\end{align*}

\end{prop}

\begin{proof}
The proof is a modification of Proposition 3.3 in \cite{LS}, see also \cite{GV}. We have obtained the global $C^1$ estimates and boundary $C^2$ estimates on $u$. Now suppose the maximum of $|\nabla^2u|_g$ is achieved at a point in the interior.

Denote $S(TM)$ the unit tangent bundle of $(M,g)$. We define a function $h:\,S(TM)\times[0,T_0)\to\mathbb{R}$, such that
\begin{align*}
h(x,e_x,t)=(\nabla^2u+m|\nabla u|^2g)(e_x,e_x),
\end{align*}
for any $x\in M$, $t\in[0,T_0)$ and $e_x\in ST_xM$, with $m>1$ a constant to be fixed. Suppose there exist $(q, t_1)\in M^{\circ}\times [0,T_0)$ and a unit tangent vector $e_q\in ST_{q}M$ such that
\begin{align*}
h(q,e_q,t_1)=\sup_{S(TM)\times[0,t_1]}h.
\end{align*}
Notice that on $S(TM)\subseteq S(TM_1)$ (here $(M_1,g_1)$ is the extension of $(M,g)$ as in Section \ref{section_sub1.1}), we can find a uniform constant $C'>0$ and a uniform small constant $\delta_0>0$ such that for any $x\in M$ and any $e_x \in T_xM_1$, $e_x$ can be extended to a unit vector field $e$ on $B_{\delta_0}(x)\subseteq M_1^{\circ}$ such that $\nabla e(x)=0$ and $|\nabla^2e|(x)\leq C'$ at this point $x$. Take the geodesic normal coordinates $(x^1,...,x^n)$ at $q$, and hence we have $\Gamma_{ab}^c(q)=0$ and $g_{ij}(q)=\delta_{ij}$. By rotating, we assume $\nabla^2u=u_{x^ix^j}$ is diagonal at $q$ and $e_q=\frac{\partial}{\partial x^1}$ at $(q,t_1)$. Let the unit vector field $e=\sum_i\xi^i\frac{\partial}{\partial x^i}$ be the extension of $e_q$ on $B_{\delta_0}(q)$ with $\nabla e(q)=0$ and $|\nabla^2 e|(q)\leq C'$. We have
\begin{align*}
\xi^1(q)=1,\,\,\xi^i(q)=0,\,i\geq1,\,\,\,\,\text{and}\,\,\,
\frac{\partial}{\partial x^i}\xi^j(q)=0,\,\,\,i,j=1,...,n.
\end{align*}
It is clear that the fact $\bar{\nabla}^2u\in \Gamma_{k}^+$ and the uniform bound of $|\nabla u|$ on $M\times[0,T_0)$ imply that there exists a uniform constant $C>-\infty$ such that $\nabla_1^2u>C$ at $(q,t_1)$.  Now we define a function $\tilde{h}$ in a small neighborhood $U\times [t_1-\epsilon, t_1+\epsilon]$ of $(q,t_1)$ such that
\begin{align*}
\tilde{h}(x,t)=(\nabla^2u+m|\nabla u|^2g)(e,e)=\xi^i\xi^j(u_{x^ix^j}-\Gamma_{ij}^au_{x^a})+m|\nabla u|^2.
\end{align*}

Since $\tilde{h}$ achieves its maximum in $U\times[t_1-\epsilon, t_1]$ at $(q,t_1)$, we have that at $(q,t_1)$,
\begin{align}
\label{inequn_innertd1}&\frac{\partial}{\partial t}\tilde{h}=u_{x^1x^1t}+2mu_{x^a}u_{x^at}\geq0,\\
&\label{inequn_innerxid1}\tilde{h}_{x^i}=u_{x^1x^1x^i}-\frac{\partial}{\partial x^i}\Gamma_{11}^au_{x^a}+2mu_{x^ax^i}u_{x^a}=0,\\
&0\geq \tilde{h}_{x^ix^j}=u_{x^1x^1x^ix^j}-\frac{\partial^2}{\partial x^j \partial x^i}\Gamma_{11}^au_{x^a}-\frac{\partial}{\partial x^i}\Gamma_{11}^au_{x^ax^j}-\frac{\partial}{\partial x^j}\Gamma_{11}^au_{x^ax^i}+m\frac{\partial^2}{\partial x^i \partial x^j}g^{ab}u_{x^a}u_{x^b}\nonumber\\
&+2mu_{x^ax^ix^j}u_{x^a}+2mu_{x^ax^i}u_{x^ax^j}+2\frac{\partial^2}{\partial x^i \partial x^j}\xi^au_{x^ax^1},\nonumber
\end{align}
where the last inequality means the Hessian of $\tilde{h}$ is non-positive. Contracting the Hessian of $\tilde{h}$ and the positively definite tensor $\bar{Q}_{ij}\equiv\frac{1}{\sigma_k(\bar{\nabla}^2u)}((n-2)\bar{T}_{ij}+\text{tr}(\bar{T}_{k-1})g_{ij})$ we have at $(q,t_1)$
\begin{align}
0\geq &\,\,\bar{Q}_{ij}u_{x^1x^1x^ix^j}-\bar{Q}_{ij}\frac{\partial^2}{\partial x^j x^i}\Gamma_{11}^au_{x^a}-2\bar{Q}_{ij}\frac{\partial}{\partial x^j}\Gamma_{11}^au_{x^ax^i}+m\bar{Q}_{ij}\frac{\partial^2}{\partial x^i \partial x^j}g^{ab}u_{x^a}u_{x^b}\nonumber\\
&\label{inequn_2d2d2}+2m\bar{Q}_{ij}u_{x^ax^ix^j}u_{x^a}+2m\bar{Q}_{ij}u_{x^ax^i}u_{x^ax^j}+2\bar{Q}_{ij}\frac{\partial^2}{\partial x^i \partial x^j}\xi^au_{x^ax^1}.
\end{align}
Differentiating equation $(\ref{equn_flowsigmak})$ with respect to $x^a$ yields
\begin{align*}
2ku_{x^at}+2ku_{x^a}=\frac{1}{\sigma_k(\bar{\nabla}^2u)}\bar{T}_{ij}[&-\nabla_aRic_{ij}+(n-2)\nabla_a\nabla_{ij}^2u+(\Delta u)_{x^a}g_{ij}\\
&+(n-2)(2\nabla_a\nabla_b u\nabla_b ug_{ij}-2\nabla_a\nabla_i u\nabla_ju)].
\end{align*}
Define the function $F(r_{ij})=\log(\sigma_k(r_{ij}))$ on $\Gamma_k^+$. Differentiating $(\ref{equn_flowsigmak})$ twice, we obtain
\begin{align*}
2k\nabla_1^2u_{t}&=(\frac{\partial^2F}{\partial r_{ab} \partial r_{ij}})\nabla_1(\bar{\nabla}^2u)_{ab}\nabla_1(\bar{\nabla}^2u)_{ij}+\frac{1}{\sigma_k(\bar{\nabla}^2u)}\bar{T}_{ij}[-\nabla_1^2Ric_{ij}+(n-2)\nabla_1^2\nabla_{ij}^2u
+\nabla_1^2(\Delta u)g_{ij}\\
&+2(n-2)((<\nabla_1^2\nabla u,\nabla u>+\nabla_1\nabla_a u\nabla_1\nabla_a u)g_{ij}-\nabla_1^2\nabla_i u\nabla_ju-\nabla_1\nabla_i u\nabla_1\nabla_ju)]-2k\nabla_1^2u\\
&\leq\frac{1}{\sigma_k(\bar{\nabla}^2u)}\bar{T}_{ij}[2(n-2)\big( \,(<\nabla_1^2\nabla u,\nabla u>+\nabla_1\nabla_a u\nabla_1\nabla_a u)g_{ij}-\nabla_1^2\nabla_i u\nabla_ju-\nabla_{1i}^2 u\nabla_{1j}^2u\big)\\
&-\nabla_1^2Ric_{ij}+(n-2)\nabla_1^2\nabla_{ij}^2u
+\nabla_1^2(\Delta u)g_{ij}]-2k\nabla_1^2u,
\end{align*}
since $F$ is concave on $\Gamma_k^+$. In particular, at $(q,t_1)$ we rewrite these two derivatives as
\begin{align}\label{inequn_flowtwoderivativesd12}
2k(u_{x^at}&+u_{x^a})=\bar{Q}_{ij}(u_{x^ix^jx^a}-\frac{\partial}{\partial x^a}\Gamma_{ij}^bu_{x^b})+\frac{\bar{T}_{ij}}{\sigma_k(\bar{\nabla}^2u)}[-\nabla_aRic_{ij}+2(n-2)(u_{x^ax^b}u_{x^b}g_{ij}-u_{x^ix^a}u_{x^j})],\\
2ku_{x^1x^1t}&\leq \frac{\bar{T}_{ij}}{\sigma_k(\bar{\nabla}^2u)}[2(n-2)\big( \,(u_{x^1x^1x^a}u_{x^a}-\frac{\partial}{\partial x^1}\Gamma_{1a}^bu_{x^b}u_{x^a}+u_{x^1x^a}u_{x^1x^a})\delta_{ij}-u_{x^1x^1x^i}u_{x^j}+\frac{\partial}{\partial x^1}\Gamma_{1i}^bu_{x^b}u_{x^j}\nonumber\\
- u_{x^1x^i}&u_{x^1x^j}\big)-\nabla_1^2Ric_{ij}]+\bar{Q}_{ij}[u_{x^ix^jx^1x^1}-\frac{\partial^2}{\partial (x^1)^2}\Gamma_{ij}^au_{x^a}-2\frac{\partial}{\partial x^1}\Gamma_{ij}^au_{x^1x^a}-2\frac{\partial}{\partial x^1}\Gamma_{1i}^au_{x^ax^j}]-2ku_{x^1x^1},\nonumber
\end{align}
and hence combining with $(\ref{inequn_2d2d2})$, we have
\begin{align*}
0\geq &\,\,\bar{Q}_{ij}(\frac{\partial^2}{\partial (x^1)^2}\Gamma_{ij}^au_{x^a}+2\frac{\partial}{\partial x^1}\Gamma_{ij}^au_{x^ax^1}+2\frac{\partial}{\partial x^1}\Gamma_{1i}^au_{x^ax^j}-\frac{\partial^2}{\partial x^j \partial x^i}\Gamma_{11}^au_{x^a}-2\frac{\partial}{\partial x^j}\Gamma_{11}^au_{x^ax^i})\\
&-\frac{\bar{T}_{ij}}{\sigma_k(\bar{\nabla}^2u)}[2(n-2)\big( \,(u_{x^1x^1x^a}u_{x^a}-\frac{\partial}{\partial x^
1}\Gamma_{1a}^bu_{x_b}u_{x^a}+u_{x^1x^a}u_{x^1x^a})\delta_{ij}-u_{x^1x^1x^
i}u_{x^j}- u_{x^1x^i}u_{x^1x^j}\\
&+\frac{\partial}{\partial x^1}\Gamma_{1i}^bu_{x^b}u_{x^j})
-\nabla_1^2\text{Ric}_{ij}]+m\bar{Q}_{ij}(\frac{\partial^2}{\partial x^i \partial x^j}g^{ab}u_{x^a}u_{x^b}
+2\frac{\partial}{\partial x^a}\Gamma_{ij}^bu_{x^b}u_{x^a}+2u_{x^ax^i}u_{x^ax^j}) \\
&+2mu_{x^a}(2ku_{x^at}+2ku_{x^a}+\frac{\bar{T}_{ij}}{\sigma_k(\bar{\nabla}^2u)}(\nabla_aRic_{ij}-2(n-2)(u_{x^ax^b}u_{x^b}g_{ij}-u_{x^ax^i}u_{x^j})))\\
&+2k(u_{x^1x^1t}+u_{x^1x^1})+2\bar{Q}_{ij}\frac{\partial^2\xi^a}{\partial x^i \partial x^j}u_{x^ax^1}.
\end{align*}
Therefore, by $(\ref{inequn_innertd1})$ and $(\ref{inequn_innerxid1})$ we have
\begin{align*}
0\geq &\,\,\bar{Q}_{ij}(\frac{\partial^2}{\partial (x^1)^2}\Gamma_{ij}^au_{x^a}+2\frac{\partial}{\partial x^1}\Gamma_{ij}^au_{x^ax^1}+2\frac{\partial}{\partial x^1}\Gamma_{1i}^au_{x^ax^j}-\frac{\partial^2}{\partial x^j \partial x^i}\Gamma_{11}^au_{x^a}-2\frac{\partial}{\partial x^j}\Gamma_{11}^au_{x^ax^i})-4kmu_{x^a}u_{x^at}\\
&+2ku_{x^1x^1}-\frac{\bar{T}_{ij}}{\sigma_k(\bar{\nabla}^2u)}[2(n-2)\big( \,(\frac{\partial}{\partial x^a}\Gamma_{11}^bu_{x^b}u_{x^a}-2mu_{x^bx^a}u_{x^b}u_{x^a}-\frac{\partial}{\partial x^1}\Gamma_{1a}^bu_{x^b}u_{x^a}+u_{x^1x^a}u_{x^1x^a})\delta_{ij}\\
&-\frac{\partial}{\partial x^i}\Gamma_{11}^bu_{x^b}u_{x^j}+2mu_{x^bx^i}u_{x^b}u_{x^j}- u_{x^1x^i}u_{x^1x^j}
+\frac{\partial}{\partial x^1}\Gamma_{1i}^bu_{x^b}u_{x^j})-\nabla_1^2\text{Ric}_{ij}]\\
&+m\bar{Q}_{ij}(2R_{iajb}u_{x^a}u_{x^b}+2u_{x^ax^i}u_{x^ax^j})+2\bar{Q}_{ij}\frac{\partial^2\xi^a}{\partial x^i \partial x^j}u_{x^ax^1}\\
&+2mu_{x^a}\big(2ku_{x^at}+2ku_{x^a}+\frac{\bar{T}_{ij}}{\sigma_k(\bar{\nabla}^2u)}(\nabla_aRic_{ij}-2(n-2)(u_{x^ax^b}u_{x^b}g_{ij}-u_{x^ax^i}u_{x^j}))\big)\\
=&\,\,\bar{Q}_{ij}((\nabla_1R_{ai1j}-\nabla_iR_{a1j1})u_{x^a}-2R_{a1j1}u_{x^ax^i})+2ku_{x^1x^1}+2\bar{Q}_{ij}(\nabla_{ij}^2\xi^a+R_{ai1j})u_{x^ax^1}\\
&-\frac{\bar{T}_{ij}}{\sigma_k(\bar{\nabla}^2u)}[2(n-2)\big( \,(R_{b1a1}u_{x^b}u_{x^a}-2mu_{x^bx^a}u_{x^b}u_{x^a}+u_{x^1x^a}u_{x^1x^a})\delta_{ij}\\
&-R_{b1i1}u_{x^b}u_{x^j}+2mu_{x^bx^i}u_{x^b}u_{x^j}- u_{x^1x^i}u_{x^1x^j}
)-\nabla_1^2\text{Ric}_{ij}] +m\bar{Q}_{ij}(2R_{iajb}u_{x^a}u_{x^b}+2u_{x^ax^i}u_{x^ax^j}) \\
&+2mu_{x^
a}\big(2ku_{x^a}+\frac{\bar{T}_{ij}}{\sigma_k(\bar{\nabla}^2u)}(\nabla_aRic_{ij}-2(n-2)(u_{x^ax^b}u_{x^b}g_{ij}-u_{x^ax^i}u_{x^j}))\big).
\end{align*}
By assumption, we have at $(q,t_1)$, $u_{x^ix^i}\leq u_{x^1x^1}$ for $i\geq 2$ and $u_{x^ix^j}=0$ for $i\neq j$. Recall that there exists a unique $C>-\infty$ on $M\times[0,T_0)$ such that $u_{x^1x^1}=\nabla_1^2u>C$ at $(q,t_1)$ and hence, we have
\begin{align*}
0\geq &-C- Cu_{x^1x^1}-(1+m)(Cu_{x^1x^1}+C)\sum_i\bar{T}_{ii}+\frac{1}{\sigma_k(\bar{\nabla}^2u)}[(2m-2(n-2))u_{x^1x^1}^2\sum_i\bar{T}_{ii}\\
&+2(n-2)(1+m)u_{x^1x^i}u_{x^1x^j}\bar{T}_{ij}]\\
\geq &-C- Cu_{x^1x^1}-(1+m)(Cu_{x^1x^1}+C)\sum_i\bar{T}_{ii}+\frac{1}{\sigma_k(\bar{\nabla}^2u)}(2m-2(n-2))u_{x^1x^1}^2\sum_i\bar{T}_{ii},
\end{align*}
where $C>0$ is a uniform constant on $M\times[0,T_0)$ depending on $k,\,n,\,C'$, $(M,g)$ and
\begin{align*}
\sup_{M\times[0,T_0)}(|u|+|u_t|+|\nabla u|+|Rm|+|\nabla Rm|+|\nabla^2Ric|).
\end{align*}
Now take $m$ to be a constant strictly larger than $(n-2)$. Recall that $\sigma_k(\bar{\nabla}^2u)$ is uniformly bounded from above and below. On the other hand, by $(\ref{inequn_structurebounda1})$, $\sum_i\bar{T}_{ii}>C$ for some uniform constant $C>0$ on $M\times[0,T_0)$, and hence we obtain that there exists a uniform constant $C>0$ on $M\times[0,T_0)$, such that
\begin{align*}
u_{x^1x^1}\leq C
\end{align*}
at $(q,t_1)$.  Therefore, combining with the boundary $C^2$ estimates, we have that there exists a uniform constant $C>0$ on $M\times[0,T_0)$, such that
\begin{align*}
|\nabla^2 u|\leq C
\end{align*}
on $M\times[0,T_0)$.

\end{proof}

Remark. Here we give a way to extend the unit vector $e_q$ at $q\in M \subseteq M_1$ in Proposition \ref{prop_global apestimatesd2} to a unit vector field $e$ in a neighborhood of $q$ with $|\nabla^2e|(q)\leq C'$ for some $C'>0$ independent of $q\in M$. Under the normal coordinates $(x^1,...,x^n)$ in $B_{\delta}(q)$ at $q$, $\Gamma_{ij}^m(0)=0$ and $g_{ij}(0)=\delta_{ij}$. Let $\tilde{e}(x)=\frac{\partial}{\partial x^1}$ for $x\in B_{\delta}(0)$, where $\delta>0$ is less than the uniform lower bound of the injectivity radius of the points $q\in M$ in $(M_1,g_1)$. Let $e(x)\equiv \xi^j\frac{\partial}{\partial x^j}=\frac{\tilde{e}(x)}{|\tilde{e}(x)|_g}$ for $x\in B_{\delta}(q)$. Since
\begin{align*}
\nabla_i \tilde{e}^j\big|_{x=0}=\frac{\partial \tilde{e}^j}{\partial x^i}=0
\end{align*}
at $x=0$ (at $q$), we have
\begin{align*}
\nabla_i\xi^j=\partial_i(\frac{\tilde{e}^j}{|\tilde{e}|_g})=\frac{\partial_i\tilde{e}^j}{|\tilde{e}|}-\frac{\partial_i\tilde{e}^a\tilde{e}_a\tilde{e}^j}{|\tilde{e}|^3}=0,
\end{align*}
at the point $q$. Therefore, the extension $\xi$ of $e_q$ in $B_{\delta}(q)$ is a unit vector field with $\nabla_i\xi^j(q)=0$. It is easy to see that there exists a uniform constant $C>0$ depending on the lower bound of the injectivity radius and upper bound of the norm of the curvature for points in $M$ in $(M_1,g_1)$, such that $|\nabla^2\xi(q)|\leq C$, for the extension $e$ of $e_q$ defined above.

\section{Convergence of the $\sigma_k$-Ricci curvature flow}\label{section_convergenceoftheflow}

Now we can prove the long time existence of the flow.
\begin{thm}\label{thm_longtimeexistence}
Assume $(M^n, g)$ is  a compact manifold with boundary of $C^{4,\alpha}$, and $(M,g)$ is either a compact domain in $\mathbb{R}^n$ or with Ricci curvature
 $Ric_g\leq -\delta_0g$ for some $\delta_0\geq (n-1)$. Assume $u_0\in C^{4,\alpha}(M)$ is a subsolution to $(\ref{equn_sigmak})$
    satisfying $(\ref{inequn_c4compatiblebdincreasing})$ at the points $x\in \partial M$ where $v(x)=0$.
    Also, assume $\phi\in C^{4+\alpha,2+\frac{\alpha}{2}}(\partial M\times[0,T_1])$ for all $T_1>0$, $\phi_t(x,t)\geq 0$ on $\partial M\times[0,+\infty)$ and
     $\phi$ satisfies the compatible condition $(\ref{equn_compatibleequations})$ with $u_0$. There exists a unique solution $u\in C^{4,2}(M\times [0,+\infty))$ to the Cauchy-Dirichlet problem $(\ref{equn_flowsigmak})-(\ref{equn_flowboundarydata})$ such that
      $u\in C^{4+\alpha,2+\frac{\alpha}{2}}(M\times[0,T])$ for all $T>0$, and the equation $(\ref{equn_flowsigmak})$ is uniformly parabolic in $t\in[0,T]$ for any $T>0$.
\end{thm}
\begin{proof}
Since $u_0$ is a subsolution to $(\ref{equn_sigmak})$, the equation is strictly parabolic at $t=0$. By the compatibility condition of $\phi$ and $u_0$, the implicit function theorem yields that there exists $T_0>0$ such that the flow is parabolic on $M\times[0,T_0)$ and the Cauchy-Dirichlet problem has a unique solution $u\in C^{4,2}(M\times[0,T_0))$ such that $u\in C^{4+\alpha,2+\frac{\alpha}{2}}(M\times[0,t_1])$ for any $t_1\in(0,T_0)$. Recall that
\begin{align*}
\sigma_k(\bar{\nabla}^2u)=\bar{\beta}_{k,n} e^{2ku_t+2ku}\geq \bar{\beta}_{k,n} e^{2ku},
\end{align*}
with the right hand side increasing by Lemma \ref{Lem_monotonicity}. Also, Lemma \ref{Lem_monotonicity} gives the uniform upper and lower bounds of $u$ on $M\times[0,T_0)$. By the a priori estimates in Lemma \ref{lem_c1estimatesonT} and Proposition \ref{prop_global apestimatesd2}, we have $\bar{\nabla}^2u\in \Gamma_k^+$ and the equation is uniformly parabolic, and hence Krylov Theorem for fully nonlinear parabolic equations yields uniform $C^{2, \alpha_{T_0}}(M)$ estimates on $u$ with some constant $0<\alpha_{T_0}<1$ for $t\in[0,T_0)$, see \cite{Krylov}. In turn the Schauder estimates yield uniform $C^{4+\alpha,2+\frac{\alpha}{2}}$ estimates on $u$ in $M\times[0,T_0)$.
Also, these a priori estimates apply to $u$ on $M\times[0,T]$ for any $T>0$ with the corresponding constants depending on $T$, and classical parabolic equation theory applies to extend the flow to $M\times[0,+\infty)$ and $u\in C^{4+\alpha, 2+\frac{\alpha}{2}}(M\times[0,T])$ for all $T>0$.
  This completes the proof of the long time existence of the flow.

\end{proof}

To show the convergence of the flow, we establish the $C^1$ and $C^2$ interior estimates on $u$ based on the bound $\sup_{U\times[0,+\infty)}|u|$ for any compact subset $U\subseteq M^{\circ}$. 

\begin{lem}\label{lem_interiorc1estimates2}
Assume $u\in C^{4,2}(M\times[0,+\infty))$ is a solution to the Cauchy-Dirichlet boundary value problem of the equation $(\ref{equn_sigmak})$ with $u_t\geq0$. Assume that for any compact subset $U\subseteq M^{\circ}$, there exists a constant $C_0=C_0(U)>0$ such that
\begin{align*}
|u|\leq C_0
\end{align*}
on $U\times[0,+\infty)$. Also, for some $T>0$, we assume that there exists a constant $C=C(T)>0$ such that
\begin{align*}
|u|+|\nabla u|\leq C(T)
\end{align*}
on $M\times[0,T]$. Then for a point $q_1\in M^{\circ}$,  there exists a constant $C_1>0$ depending on $B_{\frac{3r}{4}}(q_1)$, $C_0(B_{\frac{3r}{4}}(q_1))$ and $C(T)$ such that
\begin{align*}
|\nabla u|\leq C_1
\end{align*}
on $B_{\frac{r}{2}}(q_1)\times[0,+\infty)$, where $r$ is the distance of $q_1$ to $\partial M$.
\end{lem}
\begin{proof}
It is a modification of the interior estimates in \cite{Guan1}. For any $T_1> T$, we consider the function
\begin{align*}
F(x,t)=\mu(x)we^{f(u)}
\end{align*}
on $B_r(q_1)\times[0,T_1]$, where $w=\frac{|\nabla u|^2}{2}$, and $\mu \in C_0^2(B_{\frac{3r}{4}}(q_1))$ is a cut-off function such that
\begin{align}\label{inequn_conditioncutofffunction}
\mu=1\,\,\text{on}\,\,B_{\frac{r}{2}}(q_1),\,\,0\leq \mu\leq 1,\,\,|\nabla \mu|\leq b_0\mu^{\frac{1}{2}},\,\,|\nabla^2\mu|\leq b_0,
\end{align}
for some $b_0>0$ as defined in \cite{Guan1}, and $f(u)$ is to be determined later. By the assumption of the lemma, if $F(x,t)$ achieves its maximum on $B_{\frac{3r}{4}}(q_1)\times[0,T_1]$ at a point $(x_0, t_0)\in B_{\frac{3r}{4}}(q_1)\times [0,T]$, then $F(x,t)$ is uniformly bounded and hence
\begin{align*}
|\nabla u|\leq C
\end{align*}
on $B_{\frac{r}{2}}(q_1)\times[0,T_1]$ with a constant $C>0$ independent of $T_1$. So from now on, we assume that there exists $(x_0,t_0)\in B_{\frac{3r}{4}}(q_1)\times(T,T_1]$ such that
\begin{align*}
F(x_0,t_0)=\sup_{B_r(q_1)\times[0,T_1]}F.
\end{align*}
We choose the normal coordinate $(x^1,...,x^n)$ at $x_0$. Then at $(x_0,t_0)$, we have
\begin{align}
&\label{inequn_interiorgd1.1}\frac{w_t}{w}+f'u_t\geq0,\\
&\label{inequn_interiorgd1.2}\frac{\nabla\mu}{\mu}+\frac{\nabla w}{w}+f'\nabla u=0,\\
&\label{inequn_interiorgd1.3}\bar{T}_{ij}[\frac{\nabla_i\nabla_j \mu}{\mu}-\frac{\nabla_i\mu \nabla_j\mu}{\mu^2}+\frac{\nabla_i\nabla_j w}{w}-\frac{\nabla_iw\nabla_j w}{w^2}+f'\nabla_i\nabla_j u+f''\nabla_iu\nabla_ju]\leq0.
\end{align}
By $(\ref{inequn_interiorgd1.2})$ we have
\begin{align*}
\bar{T}_{ij}\frac{\nabla_iw\nabla_jw}{w^2}\leq 3\bar{T}_{ij}\frac{\nabla_i\mu\nabla_j\mu}{\mu^2}+\frac{3}{2}(f')^2\bar{T}_{ij}\nabla_iu\nabla_ju,
\end{align*}
and hence plugging this inequality and the definition of $w$ into $(\ref{inequn_interiorgd1.3})$ we have
\begin{align*}
&\frac{1}{w}\bar{T}_{ij}\nabla_{im}^2u\nabla_{jm}^2u+\bar{T}_{ij}(\frac{\nabla_{ij}^2\mu}{\mu}-4\frac{\nabla_i \mu \nabla_j \mu}{\mu^2})+\frac{1}{w}\bar{T}_{ij}\nabla_i\nabla_j\nabla_mu\nabla_mu\\
&+f'\bar{T}_{ij}\nabla_{ij}^2u+(f''-\frac{3}{2}(f')^2)\bar{T}_{ij}\nabla_iu\nabla_ju\leq0.
\end{align*}
Dropping the non-negative first term, changing the order of derivatives for the third order derivative term and by our choice of $\mu$, we have at $(x_0,t_0)$,
\begin{align*}
\frac{1}{w}\bar{T}_{ij}\nabla_m\nabla_i\nabla_ju\nabla_mu+f'\bar{T}_{ij}\nabla_{ij}^2u+(f''-\frac{3}{2}(f')^2)\bar{T}_{ij}\nabla_iu\nabla_ju&\leq (\frac{C}{\mu}+\frac{C}{2} w^{-1}|\nabla u|^2)\sum_{i}\bar{T}_{ii}\\
&= C(\frac{1}{\mu}+1)\sum_i\bar{T}_{ii},
\end{align*}
for some uniform constant $C>0$ depending on $b_0$ and $\sup|Rm|$ on $B_{\frac{3r}{4}}(q_1)$. Similar argument yields
\begin{align*}
\frac{1}{w}\nabla_m\Delta u\nabla_mu+f'\Delta u+(f''-\frac{3}{2}(f')^2)|\nabla u|^2\leq C(\frac{1}{\mu}+1).
\end{align*}
Combining these two inequalities and the equation $(\ref{equn_floweqnd1})$, we have
\begin{align*}
&2k(u_{x^it}u_{x_i}+|\nabla u|^2)\sigma_k(\bar{\nabla}^2u)-\bar{T}_{ab}\nabla_iu(-\nabla_iRic_{ab}+2(n-2)(\nabla_{ic}^2u\nabla_cug_{ab}-\nabla_{ia}^2u\nabla_bu))\\
\leq\,\,&-w[(n-2)\big(f'\bar{T}_{ij}\nabla_{ij}^2u+(f''-\frac{3}{2}(f')^2)\bar{T}_{ij}\nabla_iu\nabla_ju\big)+\big(f'\Delta u+(f''-\frac{3}{2}(f')^2)|\nabla u|^2\big)\sum_i\bar{T}_{ii}]\\
&+w(\frac{C}{\mu}+C)\sum_i\bar{T}_{ii}.
\end{align*}
Substituting $(\ref{inequn_interiorgd1.1})$, $(\ref{inequn_interiorgd1.2})$ and the following identity into this inequality
\begin{align*}
\bar{T}_{ab}\bar{\nabla}_{ab}u=\bar{T}_{ab}(-Ric_{ab}+(n-2)\nabla_{ab}^2u+\Delta u g_{ab}+(n-2)(|\nabla u|^2g_{ab}-\nabla_au\nabla_bu))=k\sigma_k(\bar{\nabla}^2u),
\end{align*}
we have at $(x_0,t_0)$,
\begin{align*}
&2k(-f'u_tw+|\nabla u|^2)\sigma_k(\bar{\nabla}^2u)-C|\nabla u|\sum_i\bar{T}_{ii}\\
&+2(n-2)w\bar{T}_{ij}[(\frac{\nabla_c\mu\nabla_cu}{\mu}+f'|\nabla u|^2)g_{ij}-(\frac{\nabla_i\mu\nabla_ju}{\mu}+f'\nabla_iu\nabla_ju)]\\
&\leq \,-w[(n-2)(f''-\frac{3}{2}(f')^2)\bar{T}_{ij}\nabla_iu\nabla_ju+(f''-\frac{3}{2}(f')^2)|\nabla u|^2\sum_i\bar{T}_{ii}]\\
&-kwf'\sigma_k(\bar{\nabla}^2u)+f'w\bar{T}_{ab}(-Ric_{ab}+(n-2)(|\nabla u|^2g_{ab}-\nabla_au\nabla_bu))+w(\frac{C}{\mu}+C)\sum_i\bar{T}_{ii}.
\end{align*}
If $w\leq 1$ at $(x_0,t_0)$, then we obtain the uniform upper bound of $|\nabla u|$. So we assume $w>1$. Multiplying $w^{-1}$ on both sides of the inequality, and by $(\ref{equn_flowsigmak})$ we obtain
\begin{align*}
&2k(-f'u_t+2+\frac{1}{2}f')\bar{\beta}_{k,n}e^{2k(u_t+u)}-C\sum_i\bar{T}_{ii}+2(n-2)\bar{T}_{ij}[\frac{\nabla_c\mu\nabla_cu}{\mu}g_{ij}-\frac{\nabla_i\mu\nabla_ju}{\mu}]\\
& +[(n-2)(f''-\frac{3}{2}(f')^2-f')\bar{T}_{ij}\nabla_iu\nabla_ju+(f''-\frac{3}{2}(f')^2+(n-2)f')|\nabla u|^2\sum_i\bar{T}_{ii}]\\
&\leq (\frac{C}{\mu}+C)\sum_i\bar{T}_{ii},
\end{align*}
at $(x_0,t_0)$, with $C>0$ depending on $\sup(|Rm|+|\nabla Ric|)$ and $b_0$, and hence we have
\begin{align*}
&2k(-f'u_t+2+\frac{1}{2}f')\bar{\beta}_{k,n}e^{2k(u_t+u)}\\
& +[(n-2)(f''-\frac{3}{2}(f')^2-f')\bar{T}_{ij}\nabla_iu\nabla_ju+(f''-\frac{3}{2}(f')^2+(n-2)f'-b_2)|\nabla u|^2\sum_i\bar{T}_{ii}]\\
&\leq C(1+\frac{1}{b_2})(\frac{1}{\mu}+1)\sum_i\bar{T}_{ii}
\end{align*}
for some $C>0$ depending on $n$, $\sup(|Rm|+|\nabla Ric|)$ and $b_0$, where we have used the Cauchy inequality and the constant $b_2>0$ is to be determined. Now we take
\begin{align*}
f(u)=(2+u-\inf_{B_{\frac{3r}{4}}(q_1)\times[0,+\infty)}u)^{-N}
\end{align*}
for some constant $N>1$ to be fixed. Therefore,
\begin{align*}
-N2^{-N-1}\leq& f'=-N(2+u-\inf_{B_{\frac{3r}{4}}(q_1)\times[0,+\infty)}u)^{-N-1}\leq -N(2+\text{osc}u)^{-N-1}<0,\\
f''-\frac{3}{2}(f')^2+3(n-2)f'=&N[(N+1)-N(2+u-\inf_{B_{\frac{3r}{4}}(q_1)\times\mathbb{R}_+}u)^{-N}-3(n-2)(2+u-\inf_{B_{\frac{3r}{4}}(q_1)\times\mathbb{R}_+}u)]\times\\
&(2+u-\inf_{B_{\frac{3r}{4}}(q_1)\times\mathbb{R}_+}u)^{-N-2}\\
\geq & N(2+u-\inf_{B_{\frac{3r}{4}}(q_1)\times[0,+\infty)}u)^{-N-2}[(1-2^{-N})N+1-3(n-2) (2+\text{osc} u)]
\end{align*}
where
\begin{align*}
\text{osc}u=\sup_{B_{\frac{3r}{4}}(q_1)\times[0,+\infty)}(u-\inf_{B_{\frac{3r}{4}}(q_1)\times[0,+\infty)}u)\leq 2\sup_{B_{\frac{3r}{4}(q_1)}\times[0,\infty)}|u|.
\end{align*}
Now we take $N>1$ large so that
\begin{align*}
f''-\frac{3}{2}(f')^2+3(n-2)f'>0,
\end{align*}
and take $b_2=(n-2)N(2+\text{osc}u)^{-N-1}$, and hence,
\begin{align}
&\frac{2 k}{C}(-f'u_t+2+\frac{1}{2}f')\bar{\beta}_{k,n}e^{2k(u_t+u)} +|\nabla u|^2\sum_i\bar{T}_{ii}\nonumber\\
&\label{inequn_innerd1withfd1}\leq C(1+\frac{1}{b_2})(\frac{1}{\mu}+1)\sum_i\bar{T}_{ii}
\end{align}
for some $C>0$ depending on $n$, $\sup|u|$, $\sup(|Rm|+|\nabla Ric|)$ and $b_0$. Notice that if $u_t<\frac{1}{2}$, since $u_t\geq0$, and $u$ and $f'(u)$ are uniformly bounded, we have for some uniform constant $C>0$,
\begin{align*}
|\nabla u|^2\sum_i\bar{T}_{ii}\leq C(1+\frac{1}{b_2})(\frac{1}{\mu}+1)\sum_i\bar{T}_{ii}+C.
\end{align*}
On the other hand, by $(\ref{inequn_structurebounda1})$,
\begin{align*}
\sum_i\bar{T}_{ii}\geq (n-k+1)\left(\begin{matrix} n\\ k-1\end{matrix}\right)\big(\left(\begin{matrix} n\\ k\end{matrix}\right)^{-1}\bar{\beta}_{k,n}e^{2ku_t+2ku}\big)^{\frac{k-1}{k}}\geq C
\end{align*}
for a uniform $C>0$ depending on $\sup |u|$, and hence we have
\begin{align*}
\mu|\nabla u|^2\leq C
\end{align*}
at $(x_0,t_0)$ for some uniform constant $C>0$ depending on $n$, $\sup|u|$, $\sup(|Rm|+|\nabla Ric|)$ and $b_0$, independent of $T_1$. For the case $u_t\geq \frac{1}{2}$ at $(x_0,t_0)$, the first term in $(\ref{inequn_innerd1withfd1})$ is positive and hence
\begin{align*}
|\nabla u|^2\sum_i\bar{T}_{ii}\leq C(1+\frac{1}{b_2})(\frac{1}{\mu}+1)\sum_i\bar{T}_{ii},
\end{align*}
and again we have
\begin{align*}
\mu|\nabla u|^2\leq C
\end{align*}
at $(x_0,t_0)$ for some uniform constant $C>0$ depending on $n$, $\sup|u|$, $\sup(|Rm|+|\nabla Ric|)$ and $b_0$, independent of $T_1$. Therefore, by the arbitrary choice of $T_1>T$,
\begin{align*}
F(x,t)\leq F(x_0,t_0)\leq 2Ce^{2^{-N}}
\end{align*}
for $(x,t)\in[0,+\infty)$. In particular,
\begin{align*}
|\nabla u(x,t)|\leq C
\end{align*}
for $(x,t)\in B_{\frac{r}{2}}(q_1)\times[0,+\infty)$, for some uniform constant $C>0$ depending on $n$, $\displaystyle\sup_{B_{\frac{3r}{4}}(q_1)\times[0,+\infty)}|u|$, $\displaystyle\sup_{M}(|Rm|+|\nabla Ric|)$, $b_0$ and $B_{\frac{3r}{4}}(q_1)$. Therefore, for any compact subsets $U$ and $U_1$ such that $U\subseteq U_1^{\circ}\subseteq U_1\subseteq M^{\circ}$, there exists a uniform constant $C>0$ depending on $U$, $\displaystyle\sup_{U_1\times[0,+\infty)}|u|$ and $\displaystyle\sup_{M}(|Rm|+|\nabla Ric|)$ such that
\begin{align*}
|\nabla u(x,t)|\leq C+\sup_{U\times[0,T]}|\nabla u|
\end{align*}
for $(x,t)\in U\times[0,+\infty)$.

\end{proof}

Based on the interior $C^1$ estimates, the interior $C^2$ estimates are relatively easy modifications of the $C^2$ estimates in Proposition \ref{prop_global apestimatesd2}.

\begin{lem}\label{lem_interiorc2estimates2}
Assume $u\in C^{4,2}(M\times[0,+\infty))$ is a solution to the Cauchy-Dirichlet boundary value problem of the equation $(\ref{equn_sigmak})$ with $u_t\geq0$. Assume that for any compact subset $U\subseteq M^{\circ}$, there exists a constant $C_0(U)>0$ such that
\begin{align*}
|u|\leq C_0
\end{align*}
on $U\times[0,+\infty)$. Also, for some $T>0$, we assume that there exists a constant $C=C(T)>0$ such that
\begin{align*}
|\nabla^2u|\leq C(T)
\end{align*}
on $M\times[0,T]$. Then for a point $q_1\in M^{\circ}$,  there exists a constant $C'>0$ depending on $B_{\frac{3r}{4}}(q_1)$, $C_0(B_{\frac{3r}{4}}(q_1))$ and $\displaystyle\sup_{B_{\frac{3r}{4}}(q_1)\times[0,\infty)}|\nabla u|$ such that
\begin{align*}
|\nabla^2 u|\leq C'
\end{align*}
on $B_{\frac{r}{2}}(q_1)\times[0,+\infty)$, where $r$ is the distance of $q_1$ to $\partial M$.
\end{lem}
\begin{proof}
For any $T_1> T$, we consider the function $H:\,S(TM)\times[0,T_1)\to\mathbb{R}$ such that
\begin{align*}
H(x,e_x,t)=\mu(x)h(x,e_x,t)
\end{align*}
for $x\in M$, $e_x\in ST_xM$ and $t\geq 0$, where $h$ is defined in the proof of Proposition \ref{prop_global apestimatesd2} and $\mu\in C_0^2(B_{\frac{3r}{4}}(q_1))$ satisfies $(\ref{inequn_conditioncutofffunction})$ for some constant $b_0>0$. By continuity, there exists a point $(q,t_0)\in B_{\frac{3r}{4}}(q_1)\times[0,T_1]$ and $e_{q}\in ST_{q}M$, such that
\begin{align*}
H(q,e_{q},t_0)=\sup_{STM\times[0,T_1]}\mu(x)h(x,e_x,t).
\end{align*}
If $t_0\leq T$, then by assumption, $|\nabla^2u|$ and hence $H$ are well controlled. Therefore, we assume that $t_0>T$. The same as in Proposition \ref{prop_global apestimatesd2}, we choose the normal coordinates $(x^1,...,x^n)$ at $q$ so that $e_q=\frac{\partial}{\partial x^1}$ and we extend $e_q$ to a unit vector field $e=\xi^i\frac{\partial}{\partial x^i}$ in the neighborhood of $q$ in the same way. We define the function
\begin{align*}
\tilde{H}(x,t)=H(x,e(x),t)=\mu(x)\tilde{h}(x,t)=\mu(x)(\xi^i\xi^j\nabla_i\nabla_ju+m|\nabla u|^2)
\end{align*}
in a neighborhood of $(q,t_0)$, for some constant $m>1$ to be fixed. Therefore, at $(q,t_0)$, we have
\begin{align}
&\label{equn_interiorc2d1c1}\tilde{h}_t=\nabla_1\nabla_1u_t+2m\nabla_au_t\nabla_au\geq0,\\
&\label{equn_interiorc2d1c2}\frac{\nabla \mu}{\mu}+\frac{\nabla \tilde{h}}{\tilde{h}}=0,\\
&\bar{T}_{ij}[\frac{\nabla_{ij}^2\mu}{\mu}-\frac{\nabla_i\mu\nabla_j\mu}{\mu^2}+\frac{\nabla_{ij}^2\tilde{h}}{\tilde{h}}-\frac{\nabla_i\tilde{h}\nabla_j\tilde{h}}{\tilde{h}^2}+\nabla_j\nabla_i\xi^a\nabla_{a1}^2u]\leq0,\nonumber\\
&\frac{\Delta \mu}{\mu}-\frac{|\nabla \mu|^2}{\mu^2}+\frac{\Delta\tilde{h}}{\tilde{h}}-\frac{|\nabla \tilde{h}|^2}{\tilde{h}^2}+\Delta \xi^a\nabla_{a1}^2u\leq0.\nonumber
\end{align}
Direct calculation and changing order of derivatives yield at $(q,t_0)$,
\begin{align*}
&\nabla_i\tilde{h}=\nabla_1\nabla_1\nabla_iu+Rm\ast \nabla u+2m\nabla_i\nabla_au\nabla_au,\\
&\nabla_j\nabla_i\tilde{h}=\nabla_1\nabla_1\nabla_j\nabla_i u+\nabla Rm\ast \nabla u+Rm\ast \nabla^2u+2m(\nabla_a\nabla_j\nabla_iu\nabla_au+\nabla_{ja}^2u\nabla_{ia}^2u+Rm\ast \nabla u \ast \nabla u),
\end{align*}
and hence combining these inequalities at the maximum point $(q,t_0)$ we have
\begin{align*}
&\bar{T}_{ij}[(n-2)\nabla_1\nabla_1\nabla_i\nabla_ju+\nabla_1\nabla_1\Delta ug_{ij}]\\
\leq\,\,& \bar{T}_{ij}[(n-2)\nabla_{ji}^2\tilde{h}+\Delta \tilde{h}g_{ij}]-2m[(n-2)\bar{T}_{ij}\nabla_a\nabla_j\nabla_iu\nabla_au+\nabla_a\Delta u\nabla_au\sum_i\bar{T}_{ii}]\\
&-2m[(n-2)\bar{T}_{ij}\nabla_{ja}^2u\nabla_{ia}^2u+\nabla_{ba}^2u\nabla_{ba}^2u\sum_i\bar{T}_{ii}]+(C+C|\nabla^2 u|)\sum_i\bar{T}_{ii}\\
\leq\,\,& -\tilde{h}\bar{T}_{ij}[(n-2)(\frac{\nabla_{ij}^2\mu}{\mu}-2\frac{\nabla_i\mu\nabla_j\mu}{\mu^2})+(\frac{\Delta \mu}{\mu}-2\frac{|\nabla \mu|^2}{\mu^2}) g_{ij}]+(C+C|\nabla^2 u|)\sum_i\bar{T}_{ii}\\
&-2m[(n-2)\bar{T}_{ij}\nabla_{aji}u\nabla_au+\nabla_a\Delta u\nabla_au\sum_i\bar{T}_{ii}]-2m[(n-2)\bar{T}_{ij}\nabla_{ja}^2u\nabla_{ia}^2u+\nabla_{ba}^2u\nabla_{ba}^2u\sum_i\bar{T}_{ii}]\\
\leq\,\,& -2m[(n-2)\bar{T}_{ij}\nabla_{aji}u\nabla_au+\nabla_a\Delta u\nabla_au\sum_i\bar{T}_{ii}]-2m[(n-2)\bar{T}_{ij}\nabla_{ja}^2u\nabla_{ia}^2u+\nabla_{ba}^2u\nabla_{ba}^2u\sum_i\bar{T}_{ii}]\\
&+C(1+(1+\frac{1}{\mu})|\nabla^2 u|)\sum_i\bar{T}_{ii},
\end{align*}
where $C$ depends on $\sup|Rm|$, $b_0$, $\displaystyle\sup_{B_{\frac{3r}{4}}(q_1)\times[0,\infty)}|\nabla u|$ and the uniform upper bound of $|\nabla^2e|(q)$ (see Proposition \ref{prop_global apestimatesd2}), and hence combining this inequality with the two inequalities $(\ref{inequn_flowtwoderivativesd12})$ we have
\begin{align*}
&2k(\nabla_{11}^2u+\nabla_{11}^2u_t)\sigma_k(\bar{\nabla}^2u)-2(n-2)\bar{T}_{ij}[(\nabla_{a11}u\nabla_au+\nabla_{1a}^2u\nabla_{1a}^2u)g_{ij}-\nabla_{i11}u\nabla_ju-\nabla_{1i}^2 u\nabla_{1j}^2u]\\
\leq \,\,&-4km(\nabla_au_t\nabla_au+|\nabla u|^2)\sigma_k(\bar{\nabla}^2u)-2m[(n-2)\bar{T}_{ij}\nabla_{ja}^2u\nabla_{ia}^2u+\nabla_{ba}^2u\nabla_{ba}^2u\sum_i\bar{T}_{ii}]\\
&+C(1+m+(1+m+\frac{1}{\mu})|\nabla^2 u|)\sum_i\bar{T}_{ii}.
\end{align*}
Plugging in $(\ref{equn_interiorc2d1c1})$ and $(\ref{equn_interiorc2d1c2})$, we have
\begin{align*}
&2k\nabla_{11}^2u\sigma_k(\bar{\nabla}^2u)-2(n-2)\bar{T}_{ij}[\nabla_{1a}^2u\nabla_{1a}^2ug_{ij}-\nabla_{1i}^2 u\nabla_{1j}^2u]\\
\leq \,\,&-4km|\nabla u|^2\sigma_k(\bar{\nabla}^2u)-2m[(n-2)\bar{T}_{ij}\nabla_{ja}^2u\nabla_{ia}^2u+\nabla_{ba}^2u\nabla_{ba}^2u\sum_i\bar{T}_{ii}]\\
&+C(1+m+(1+m+\frac{1}{\mu})|\nabla^2 u|)\sum_i\bar{T}_{ii}.
\end{align*}
Since $\nabla_{1i}^2u(q,t_0)=0$ for $i\geq 2$ by the choice of coordinates as in Proposition \ref{prop_global apestimatesd2}, and
\begin{align*}
\nabla_{11}^2u(q,t_0)\geq \nabla_{ii}^2u(q,t_0)
\end{align*}
for $i\geq 2$, and hence we have
\begin{align*}
&2k(\nabla_{11}^2u+2m|\nabla u|^2)\sigma_k(\bar{\nabla}^2u)+(2m-2(n-2))\nabla_{11}^2u\nabla_{11}^2u\sum_i\bar{T}_{ii}\\
\leq \,\,&C(1+m+n(1+m+\frac{1}{\mu})|\nabla_{11}^2 u|)\sum_i\bar{T}_{ii}.
\end{align*}
We take $m$ large and  use the equation $(\ref{equn_flowsigmak})$ to obtain
\begin{align*}
&2k(\nabla_{11}^2u+2m|\nabla u|^2)\bar{\beta}_{k,n}e^{2k(u+u_t)}+\nabla_{11}^2u\nabla_{11}^2u\sum_i\bar{T}_{ii}\\
\leq \,\,&C(1+(1+\frac{1}{\mu})|\nabla_{11}^2 u|)\sum_i\bar{T}_{ii},
\end{align*}
for some uniform $C>0$ independent of $T_1$, and hence if $\nabla_{11}^2u(q,t_0)>1$, the first term in this inequality is positive and since $\sum_i\bar{T}_{ii}$ is uniformly bounded from below by $(\ref{inequn_structurebounda1})$, we have
\begin{align*}
\mu\nabla_{11}^2u(q,t_0)\leq \,\,C,
\end{align*}
for some uniform constant $C>0$ independent of $T_1$, and hence
\begin{align*}
\tilde{H}\leq C
\end{align*}
in $B_{\frac{3r}{4}}(q_1)\times[0,T_1]$ with $C>0$ independent of $T_1$;  while if $\nabla_{11}^2u(q,t_0)\leq1$, we trivially have the uniform upper bound of $\tilde{H}$ by its definition and the bound of $|\nabla u|$ on $B_{\frac{3r}{4}}(q_1)\times[0,\infty)$. By the arbitrary choice of $T_1>T$, $\tilde{H}$ has a uniform upper bound on $B_{\frac{3r}{4}}(q_1)\times[0,\infty)$. In particular,
\begin{align*}
\nabla_{11}^2u\leq \,\,C,
\end{align*}
in $B_{\frac{r}{2}}(q_1)\times[0,\infty)$. Since $\bar{\nabla}^2u\in \Gamma_k^+$, and $|\nabla u|$ is uniformly bounded in $B_{\frac{3r}{4}}(q_1)$, we have that there exists a uniform constant $\alpha>-\infty$ such that
\begin{align*}
\Delta u\geq \alpha,
\end{align*}
and hence
\begin{align*}
|\nabla^2 u|\leq n^3(C+|\alpha|),
\end{align*}
on $B_{\frac{r}{2}}(q_1)\times[0,\infty)$. This completes the proof of the lemma.
\end{proof}

Now we prove the convergence of the flow and the asymptotic behavior near the boundary as $t\to\infty$.

\begin{proof}[Proof of Theorem \ref{thm_main}]
Long time existence of the solution $u$ has been obtained in Theorem \ref{thm_longtimeexistence}, and we only need the consider the convergence of $u$ and its asymptotic behavior near the boundary as $t\to \infty$.

First we establish the uniform upper bound estimates on $u$ on any given compact subset of $M^{\circ}$. By the Maclaurin's inequality, $u$ is a subsolution to the $\sigma_1$-Ricci curvature flow $(\ref{equn_flowsigmak})$. By the maximum principle for the $\sigma_1$-Ricci curvature flow in Lemma \ref{lem_comparison}, to get the upper bound of $u$, it suffices to find a super-solution to the scalar curvature equation i.e., $(\ref{equn_sigmak})$ with $k=1$ satisfying $(\ref{equn_boundaryblowingup})$ near $\partial M$. Direct application of Lemma 5.2 in \cite{GSW}, where a sequence of super-solutions to the scalar curvature equation on corresponding small geodesic balls blowing up on the boundary was constructed, yields the upper bound of $u$: 
\begin{align*}
\limsup_{x\to\partial M}[u(x,t)+\log(r(x))]\leq 0,
\end{align*}
uniformly for all $t>0$; and moreover, for any compact subset $U\subseteq M^{\circ}$, there exists a constant $C>0$ depending on $U$ such that $u(x,t)\leq C$ for all $(x,t)\in U\times[0,+\infty)$. Here is an alternative argument: by maximum principle  for $\sigma_1$-Ricci curvature flow in Lemma \ref{lem_comparison},
\begin{align*}
u(x,t)\leq u_{LN}(x),
\end{align*}
for $(x,t)\in M^{\circ}\times[0,\infty)$, where $u_{LN}$ is the
 solution to the Loewner-Nirenberg problem of the constant scalar curvature equation on $M$. Recall that
\begin{align*}
u_{LN}(x)\leq -\log(r(x))+o(1)\,\,\,\text{near the boundary},
\end{align*}
with $o(1)\to 0$ as $x\to \partial M$, see in \cite{LL}\cite{LM}\cite{ACF} for instance.

By Lemma \ref{Lem_monotonicity}, $u(x,t)$ is increasing along $t>0$ and hence
\begin{align*}
u_0(x)\leq u(x,t) \leq u_{LN}(x)
\end{align*}
for $(x,t)\in M^{\circ}\times [0,+\infty)$. Or just use the super-solution to $(\ref{equn_sigmak})$ on a small ball centered at $x$ constructed in  Lemma 5.2 in \cite{GSW} instead of $u_{LN}$. Therefore, $u(x,t)$ converges as $t\to\infty$ for any $x\in M^{\circ}$. By Lemma \ref{lem_interiorc1estimates2} and Lemma \ref{lem_interiorc2estimates2},
we have that for any compact subsets $U\subseteq U_1\subseteq M^{\circ}$ with $U\subseteq U_1^{\circ}$, there exists a constant $C>0$ such that
\begin{align*}
|\nabla u|+|\nabla^2u|\leq C
\end{align*}
in $U_1\times[0,\infty)$ and hence, the equation $(\ref{equn_flowsigmak})$ is uniformly parabolic and  by $(\ref{equn_flowsigmak})$, $u_t$ has a uniform upper bound on $U_1\times[0,\infty)$.   By Krylov's Theorem and the classical Schauder estimates, we have that there exists a uniform constant $C>0$ depending on $U_1$ such that
\begin{align*}
\|u\|_{C^{4,2}(U\times[0,\infty))}\leq C,
\end{align*}
and
\begin{align}\label{inequn_c4c2boundflow}
\|u\|_{C^{4,\alpha}(U)}\leq C,
\end{align}
for all $t\geq0$. Since $u$ increases and has uniform upper bound in $U$, by the Harnack inequality of the linear uniformly parabolic equation $(\ref{equn_flowlinear1})$ for $u_t$, we have
\begin{align*}
v=u_t\to 0
\end{align*}
uniformly on $U$ as $t\to+\infty$. Therefore, $u(x,t)\to u_{\infty}(x)$ uniformly for $x\in U$ as $t\to+\infty$. By the uniform bound $(\ref{inequn_c4c2boundflow})$ and the interpolation inequality, we have
\begin{align*}
u(x,t)\to u_{\infty}(x)
\end{align*}
in $C^4(U)$ as $t\to\infty$. By the arbitrary choice of the compact subset $U \subseteq M^{\circ}$, we have that $u_{\infty}$ is a solution to $(\ref{equn_sigmak})$ in $M^{\circ}$.

Now we consider the lower bound of $u$ near the boundary. Applying Lemma \ref{lem_boundarylowerbd1} to be proved later, we have that there exist $\delta_1>0$ small and $T>0$ large, such that
\begin{align*}
u(x,t)\geq -\log(r(x)+\epsilon(t))+w(x)
\end{align*}
for $x\in M$ with $r(x)\leq \delta_1$ and $t\geq T$, where $w(x)\leq 0$ with $w\big|_{\partial M}=0$ and $\epsilon(t)\to 0$ as $t\to+\infty$. By the upper and lower bound estimates on $u$ near the boundary, we have
\begin{align*}
u_{\infty}(x)+\log(r(x))\to 0
\end{align*}
uniformly as $x\to \partial M$.
\end{proof}

We will show the lower bound of the asymptotic behavior of $u$ near the boundary as $t\to\infty$, for which we need $\phi$ to increase not too slowly.
\begin{lem}\label{lem_boundarylowerbd1}
Let $(M,g)$, $u_0$, $\phi$, $T_1>T$ and $u$ be as in Theorem \ref{thm_main}. Let $r(x)$ be the distance function of $x\in M$ to the boundary $\partial M$.
 Then there exist  $\delta_1>0$ small and $T_2>T_1$, such that
\begin{align*}
u(x,t)\geq -\log(r(x)+\epsilon(t))+w(x)
\end{align*}
for $x\in M$ with $r(x)\leq \delta_1$ and $t\geq T_2$, where $\epsilon=\xi(t)^{-1}$ and $w$ is a function of $C^2$ where $r(x)\leq \delta_1$ such that $w(x)\leq 0$ with $w\big|_{\partial M}=0$.
\end{lem}
\begin{proof}
Let $\delta_1>0$ be a small constant to be fixed. Define the exponential map $\text{Exp}:\,\partial M\times[0,\delta_1]\to M$ such that $\text{Exp}_q(s)\in M$ is the point on the geodesic starting from $q\in \partial M$ in the direction of inner normal vector with distance $s$ to $q$. $\delta_1$ is chosen small so that $\text{Exp}$ is a diffeomorphism to the image. Define
\begin{align*}
U_{\delta_1}=\{\text{Exp}_q(s)\big|\,(q,s)\in \partial M\times[0,\delta_1]\}.
\end{align*}
The metric has the orthogonal decomposition
\begin{align*}
g=ds^2+g_{s},
\end{align*}
with $g_{s}$ the restriction of $g$ on $\Sigma_{s}=\{z\in M\big|\,\,r(z)=s\}$ for $0\leq s\leq \delta_1$.  Define the function
\begin{align*}
\underline{u}(x,t)=-\log(r(x)+\epsilon(t))+w(x)
\end{align*}
for $(x,t)\in U_{\delta_1}\times[T,+\infty)$ where
\begin{align*}
w(x)=A(\frac{1}{(r(x)+\delta)^p}-\frac{1}{\delta^p})
\end{align*}
with constants $A>0$, $p>1$ large and $\delta>0$ small to be determined. By definition, we have
\begin{align}\label{inequn_slowspeedincreasing2-2}
-\frac{\epsilon'(t)}{\epsilon(t)^2}=\xi'(t)\leq \tau
\end{align}
for $t\geq T$. Let $\tilde{r}(x,t)=r(x)+\epsilon(t)$. For any $x_0\in U_{\delta_1}^{\circ}$, let $\{e_1,...,e_n\}$ be an orthonormal basis at $x_0$ such that $e_1=\frac{\partial}{\partial r}$. The same calculation as in Lemma 5.1 in \cite{GSW} yields
\begin{align*}
\bar{\nabla}^2\underline{u}=&-Ric_g+(n-2)\nabla^2w+\Delta w g+(n-2)(w')^2\left(\begin{matrix}0&&&\\
&1&&\\ &&\ddots&\\ &&&1 \end{matrix}\right)\\
&+\frac{1}{\tilde{r}^2}\left[(n-1)g-2(n-2)\tilde{r}w'\left(\begin{matrix}0&&&\\
&1&&\\ &&\ddots&\\ &&&1 \end{matrix}\right)-\tilde{r}((n-2)\nabla^2r+\Delta r g)\right].
\end{align*}
Recall that $\nabla^2r$ and $\Delta r$ are the second fundamental form and the mean curvature of $\Sigma_{r(x_0)}$, which are uniformly bounded by a constant $\gamma\geq 1$ on $U_{\delta_1}$:
\begin{align*}
\gamma g\geq (n-2)\nabla^2r+\Delta r g \geq -\gamma g.
\end{align*}
We denote the bracketed term above on the right hand side as $\Phi$. Taking $\delta+\delta_1<1$, we have
\begin{align*}
w'=-Ap(r+\delta)^{-p-1}\leq -Ap,
\end{align*}
and hence,
\begin{align*}
\frac{1}{\tilde{r}^2}\Phi\geq \frac{1}{\tilde{r}^2}\left[(n-1)g+\,\tilde{r}\left(\begin{matrix}-\gamma&&&\\
&2(n-2)Ap-\gamma&&\\ &&\ddots&\\ &&&2(n-2)Ap-\gamma \end{matrix}\right)\right].
\end{align*}
Now we let $\delta_1<\frac{1}{20\gamma}$ and choose $T'>0$ to be large so that $\epsilon(t)<\frac{1}{20\gamma}$ for $t\geq T'$. There exists $K_0>0$ such that for $Ap>K_0$, we have that
\begin{align*}
\det(\frac{1}{\tilde{r}^2}\Phi)&\geq \frac{1}{\tilde{r}^{2n}}(n-1)^n(1-\frac{\gamma \tilde{r}}{n-1})(1+\frac{2(n-2)Ap-\gamma}{n-1}\tilde{r})^{n-1}\\
&\geq \frac{1}{\tilde{r}^{2n}}(n-1)^n(1+Ap\tilde{r}).
\end{align*}
Recall that $-Ric_g\geq0$. For any large constant $\Lambda>0$, there exists  $A_0>0$ and $p_0>0$ so that for $A>A_0$ and $p\geq p_0$,
\begin{align*}
(n-2)\nabla^2w+\Delta w g\geq \Lambda g,
\end{align*}
on $U_{\delta_1}$. Therefore, if we also assume $Ap\geq 8n\tau$, then we obtain
\begin{align*}
\log(\det(\bar{\nabla}^2\underline{u}))-\log(\bar{\beta}_{n,n})-2n\underline{u}&\geq \log(\det(\frac{1}{\tilde{r}^2}\Phi))-\log(\bar{\beta}_{n,n})-2n\underline{u}\\
&\geq \log\big(\tilde{r}^{-2n}(1+Ap\tilde{r})\big)-2n\underline{u}\\
&\geq  \log(1+Ap\tilde{r})\geq \log(1+8n\tau\tilde{r}),
\end{align*}
and hence for $\tilde{r}\leq (8n\tau)^{-1}$ and $t\geq \max\{T,T'\}$, by $(\ref{inequn_slowspeedincreasing2-2})$ we have
\begin{align}\label{inequn-subsolutionnearboundary3}
\log(\det(\bar{\nabla}^2\underline{u}))-\log(\bar{\beta}_{n,n})-2n\underline{u}&\geq 4n\tau \tilde{r}\geq -2n\frac{\epsilon'}{\tilde{r}}=2n\underline{u}_t.
\end{align}
Since $\displaystyle\lim_{t\to\infty}\epsilon(t)=0$, we take $T_2\geq \max\{T_1,T'\}$ such that $\epsilon(t)\leq (16n\tau)^{-1}$ for $t\geq T_2$ and let 
\begin{align*}
\delta_1<\min\{(16n\tau)^{-1},(20\gamma)^{-1}\}.
\end{align*}
We will choose $A$ and $p$ large so that $\underline{u}$ gives a lower bound of $u$ on $U_{\delta_1}\times[T_2,\infty)$. Notice that $\partial U_{\delta_1}=\Sigma_{\delta_1}\bigcup \partial M$. By assumption we have
\begin{align*}
\underline{u}(x,t)=\log(\xi(t))\leq \phi(x,t)
\end{align*}
for $(x,t)\in \partial M\times[T_2,\infty)$. On $\Sigma_{\delta_1}$, since $u$ is increasing, we have $u(x,t)\geq u_0(x)$. Notice that there exists $A_1>0$ such that for $A\geq A_1$ and any $p\geq 1$, we have
\begin{align*}
-\log(\delta_1)+A\,((\delta_1+\delta)^{-p}-\delta^{-p})<\inf_{\Sigma_{\delta_1}}u_0,
\end{align*}
and hence we have on $\Sigma_{\delta_1}\times[T_2,\infty)$,
\begin{align*}
\underline{u}\leq u.
\end{align*}
Finally, we consider the control on $U_{\delta_1}\times\{T_2\}$. Since $\underline{u}(\cdot,T_2),\,u(\cdot,T_2)\in C^1(M)$ and $\underline{u}\leq u=\phi$ on $\partial M\times\{T_2\}$, there exist $A_2>0$ and $p_2>0$ such that for $A\geq A_2$ and $p\geq p_2$, we have
\begin{align*}
\underline{u}\leq u
\end{align*}
on $U_{\delta_1}\times\{T_2\}$.

In summary, we assume
\begin{align*}
&Ap\geq \max\{K_0,\,8n\tau\},\,\,p\geq\max\{1,\,p_0,\,p_2\},\,\,A\geq\max\{A_0,\,A_1,\,A_2\},\\
&\,\,\,\delta+\delta_1<1,\,\,\delta_1<\min\{(16n\tau)^{-1},(20\gamma)^{-1}\},
\end{align*}
and $\delta_1>0$ is small so that $\text{Exp}$ is a diffeomorphism. Therefore, $\underline{u}$ is a sub-solution to $(\ref{equn_flowsigmak})$  for $k=n$  by $(\ref{inequn-subsolutionnearboundary3})$ and hence a sub-solution to $(\ref{equn_flowsigmak})$ for $1\leq k\leq n$ on $U_{\delta_1}\times[T_2,\infty)$ , by Maclaurin's inequality; moreover,
\begin{align*}
\underline{u}\leq u,\,\,\,\text{on}\,\,(\partial M \bigcup \Sigma_{\delta_1})\times[T_2,\infty)\,\bigcup\,U_{\delta_1}\times\{T_2\}.
\end{align*}
Therefore, by the maximum principle in Lemma \ref{lem_comparison}, we have
\begin{align*}
u(x,t)\geq \underline{u}(x,t)=-\log(r(x)+\epsilon(t))+A\,((r(x)+\delta)^{-p}-\delta^{-p})
\end{align*}
on $U_{\delta_1}\times[T_2,\infty)$.

\end{proof}

\begin{proof}[Proof of Corollary \ref{cor_Dirichletboundaryvalueproblemflow}]
The equation $(\ref{equn_sigmak})$ is conformally covariant, and hence it is equivalent to consider the case when the background metric $g$ is the Euclidean metric
 when $(M,g)$ is a domain in the Euclidean space, while $g\in C^{4,\alpha}$ is chosen to be a metric constructed in \cite{Lohkamp} (see Section \ref{section_sub1.1} in the present paper)
  such that $Ric_g<-(n-1)g$ in the conformal class for a general manifold $(M,g)$. Let $u_0=\underline{u}+\min\{0,\,\inf_{\partial M}\varphi_0\}$ with $\underline{u}$ a sub-solution constructed in Section \ref{section_sub1.1} for $A>0$ and $p>0$ large,
   and when $(M,g)$ is a domain in Euclidean space one can just take $\underline{u}$ to be the global sub-solution in \cite{GSW} (just take the function $\eta(s)=s$ for the sub-solution $\underline{u}$ in Section \ref{section_sub1.1}) with $A>0$ and $p>0$ large. Then $u_0$ is a strict sub-solution near the boundary with $u_0<\varphi_0$ on $\partial M$ and hence,
   we can construct the boundary data $\phi\in C^{4+\alpha,2+\frac{\alpha}{2}}(\partial M\times[0,\infty))$  satisfying the compatible condition $(\ref{equn_compatibleequations})$ at $t=0$ such that
    $\phi_t\geq 0$ on $\partial M \times[0,\infty)$ and $\phi(x,t)\to \varphi_0(x) $ uniformly in $C^{4,\alpha'}(\partial M)$ as $t\to\infty$ for some $0<\alpha'<1$.

   Consider the Cauchy-Dirichlet boundary value problem $(\ref{equn_flowsigmak})-(\ref{equn_flowboundarydata})$. It is clear that Lemma \ref{lem_comparison} and Lemma \ref{Lem_monotonicity} still hold true.  Recall that by Maclaurin's inequality $u$ is a sub-solution to the $\sigma_1$-Ricci curvature flow
     $(\ref{equn_flowsigmak})$. On the other hand, for the $\sigma_1$-Ricci equation $(\ref{equn_sigmak})$, which is the Yamabe equation, classical variational
      methods yield a unique minimizing solution $u_1$ to the Dirichlet boundary value problem with $u_1=\varphi_0$ on $\partial M$, see \cite{ML}.
      By Lemma \ref{lem_comparison} for the $\sigma_1$-Ricci curvature flow, we have $u(x,t)\leq u_1(x)$ for $(x,t)\in M\times[0,\infty)$ and hence we have a uniform
       upper bound of $u$. Also, the a priori $C^2$ estimates from Lemma \ref{lem_gradientboundaryests} to Proposition \ref{prop_global apestimatesd2} hold with uniform bound of $\|u(\cdot,t)\|_{C^2(M)}$ independent of $t>0$. By Theorem \ref{thm_longtimeexistence}, we have the long time existence of the unique solution $u$.
   Things are even better in this case: there exists a uniform constant $C>0$ such that for any $T>0$,
   \begin{align}\label{inequn_DirichletboundaryvalueproblemC3}
   \|u\|_{C^{4+\alpha,2+\frac{\alpha}{2}}(M\times[T,T+1])}\leq C,
   \end{align}
    by Krylov's Theorem and the standard Schauder estimates. Remark that here we do not need the locally uniformly interior estimates.

  By $(\ref{inequn_DirichletboundaryvalueproblemC3})$, there exists a sequence $t_j\to \infty$, such that $u(x,t_j)\to u_{\infty}(x)$ in $C^4(M)$ for some $u_{\infty}\in C^{4,\alpha}(M)$ as $t_j\to\infty$. By monotonicity of $u$, $u(x,t)\to u_{\infty}(x)$ uniformly for $x\in M$ as $t\to\infty$. By $(\ref{inequn_DirichletboundaryvalueproblemC3})$ and the interpolation inequality, we have $u(x,t)\to u_{\infty}(x)$ uniformly in $C^4(M)$ as $t\to \infty$
       and hence, $u_{\infty}=\varphi_0$ on $\partial M$. Since $u_t\geq0$ satisfies the linear uniformly parabolic equation $(\ref{equn_flowlinear1})$,
       by Harnack inequality, $u_t\to0$ locally uniformly in $M^{\circ}$ as $t\to\infty$ and hence, $u_{\infty}$ is a solution to $(\ref{equn_sigmak})$. This completes the proof of the corollary.

 \end{proof}

 \end{document}